\documentclass[12pt]{amsart}
\usepackage[dvips]{color}
\usepackage{amsmath}
\usepackage{amsxtra}
\usepackage{amscd}
\usepackage{amsthm}
\usepackage{amsfonts}
\usepackage{amssymb}
\usepackage{eucal}
\usepackage{graphics}
\def\({\left(}
\def\){\right)}
\newcommand\bls{\boldsymbol}
\newcommand{\gl}{\mathfrak{gl}}
\newcommand{\ga}{\gamma}
\newcommand\Ref{\eqref}

\newcommand{\mub}{{\bls\mu}}
\newcommand{\omegab}{{\bls\omega}}
\newcommand{\lab}{{\bls\la}}

\newcommand{\ket}[1]{{| #1 \rangle}}      

\newcommand{\bb}{\mathbf{b}}
\newcommand{\ab}{\mathbf{a}}

\newcommand{\nn}{\nonumber}
\newcommand{\bea}{\begin{eqnarray}}
\newcommand{\ena}{\end{eqnarray}}
\def\bel{\begin{eqnarray}}
\def\enl{\end{eqnarray}}
\newcommand{\be}{\begin{eqnarray*}}
\newcommand{\en}{\end{eqnarray*}}

\newcommand{\T}{{\otimes}}

\newcommand{\C}{{\mathbb C}}
\newcommand{\Z}{{\mathbb Z}}

\newenvironment{tenumerate}{
  \begin{enumerate}
  
  }{\end{enumerate}}
\newcommand{\bi}{\begin{tenumerate}}
\newcommand{\ei}{\end{tenumerate}}
\newcommand{\isoto}[1][]%
{{\mathop{\buildrel{\sim}\over\longrightarrow}\limits_{#1}}}

\def\[{\left[}
\def\]{\right]}
\newcommand{\la}{\lambda}

\newcommand{\al}{\alpha}

\numberwithin{equation}{section}
\newtheorem{thm}{Theorem}[section]
\newtheorem{prop}[thm]{Proposition}
\newtheorem{lem}[thm]{Lemma}

\newtheorem{cor}[thm]{Corollary}

\newtheorem{conj}[thm]{Conjecture}

\newcommand{\glinf}{\mathfrak{gl}_\infty}
\newcommand{\glhf}{\mathfrak{gl}_{\infty/2}}
\newcommand{\Yb}{\mathcal{Y}}
\newcommand{\E}{{\mathcal E}}

\newcommand{\one}{1}

\def\bi{\mathbf{i}}

\definecolor{red}{rgb}{0,0,0}
\definecolor{7/17}{rgb}{0,0,0}
\definecolor{7/17comment}{rgb}{0,0,0}
\definecolor{7/29}{rgb}{0,0,0}
\definecolor{7/29comment}{rgb}{0,0,0}
\definecolor{7/18}{rgb}{0.0,0.0,0}
\definecolor{7/18comment}{rgb}{0.0,0.0,0.0}
\definecolor{zh}{rgb}{0.0,0.0,0.9} 

\begin{document}

\begin{title}[Quantum toroidal  $\mathfrak{gl}_1$ algebra : plane partitions]
{Quantum toroidal $\mathfrak{gl}_1$ algebra : plane partitions}
\end{title}
\author{B. Feigin, M. Jimbo, T. Miwa and E. Mukhin}
\address{BF: Landau Institute for Theoretical Physics,
Russia, Chernogolovka, 142432, prosp. Akademika Semenova, 1a,   \newline
Higher School of Economics, Russia, Moscow, 101000,  Myasnitskaya ul., 20 and
\newline
Independent University of Moscow, Russia, Moscow, 119002,
Bol'shoi Vlas'evski per., 11}
\email{bfeigin@gmail.com}
\address{MJ: Department of Mathematics,
Rikkyo University, Toshima-ku, Tokyo 171-8501, Japan}
\email{jimbomm@rikkyo.ac.jp}
\address{TM: Department of Mathematics,
Graduate School of Science,
Kyoto University, Kyoto 606-8502,
Japan}\email{tmiwa@kje.biglobe.ne.jp}
\address{EM: Department of Mathematics,
Indiana University-Purdue University-Indianapolis,
402 N.Blackford St., LD 270,
Indianapolis, IN 46202}\email{mukhin@math.iupui.edu}

\begin{abstract}
In 
third paper of the series we construct a large family of 
representations of the quantum toroidal $\gl_1$ 
algebra
whose bases are parameterized by plane partitions with various boundary conditions and restrictions. We study the corresponding formal characters. As an application we obtain a Gelfand-Zetlin type 
basis for a class of irreducible lowest weight $\gl_\infty$-modules.
\end{abstract}

\maketitle

\section{Introduction}

In this paper we continue our study of representations of 
an algebra $\mathcal E=\mathcal E_{q_1,q_2,q_3}$, 
depending on 
three complex parameters $q_1,q_2,q_3$ with $q_1q_2q_3=1$.
It was originally introduced by Miki \cite{Miki07} 
as a two parameter analog of the $W_{1+\infty}$ algebra.
The basic structure theory of $\E$
and its representations have
been established in \cite{Miki07}. 
Its connection with the Macdonald operator and the deformed Virasoro/$W$-algebras was 
also revealed there. 
Some of his results will be
recalled in Section \ref{preliminaries}.  
Essentially the same algebra has been rediscovered 
later on and was 
given various other names: the Ding-Iohara algebra in 
\cite{FT},\cite{FHHSY}, or 
the elliptic Hall algebra in \cite{SV1}, \cite{SV2}. 
Having been unaware of the work of Miki, we have 
called $\E$ ``quantum continuous $\gl_\infty$'' in 
\cite{FFJMM1},\cite{FFJMM2}.
We are sorry about this oversight.   
 
We have decided to call $\E$ 
the quantum toroidal $\mathfrak{gl}_1$ algebra  
for the following reason. 
Let $\mathfrak d_q$ be the algebra generated by the symbols 
$Z^{\pm1},D^{\pm1}$
satisfying $DZ=qZD$, where $q\in\C^\times$, and regard it as
a Lie algebra endowed with the Lie bracket $[a,b]=ab-ba$.
The Lie algebra $\mathfrak{d}_q$ has  
a two-dimensional central extension 
$\mathfrak{d}_{q,c_1,c_2}$. 
As mentioned in \cite{Miki07}, the algebra 
$\E_{q_1,q_2,q_3}$ is a quantization of the universal enveloping algebra 
$U\left(\mathfrak{d}_{q,c_1,c_2}\right)$, where
one of the parameters, say $q_1$, is the quantization parameter and $q_2=q$. 
The situation is 
similar to that of the quantum toroidal algebra
$U_{q}(\mathfrak{sl}_{N,tor})$, 
whose classical limit is a
central extension of the Lie algebra of 
$N\times N$ matrices $x$ with entries in $\mathfrak{d}_q$, 
such that $\mathrm{res}\,\mathrm{tr}(x)=0$. 
(Here  $\mathrm{res}(a)=a_{0,0}$  
for $a=\sum a_{i,j}Z^iD^j\in \mathfrak{d}_q$.)

Throughout this paper we shall restrict our considerations to 
representations of a quotient of $\E$
by a one-dimensional center. The classical limit 
of the quotient algebra is 
$\mathfrak{d}_{q,\kappa,0}$, 
see Section \ref{preliminaries}.

\medskip

The algebra $\mathfrak d_q$ is isomorphic 
to the algebra of $q$-difference
operators. Namely 
$\mathfrak d_q$ has a faithful representation in the 
space $V=\C[Z,Z^{-1}]$, such that
 $Z$ acts as the multiplication operator 
$f(Z)\mapsto Zf(Z)$, 
and $D$ as the $q$-difference 
operator
$f(Z)\mapsto f(qZ)$. 
This gives rise to a Lie algebra homomorphism
$\mathfrak d_{q,\kappa,0}\rightarrow \gl_{\infty,\kappa}$,
where ${\mathfrak{gl}}_{\infty,\kappa}$
is the central extension of the
Lie algebra of linear transformations $T:\ V\rightarrow V$, 
$T(Z^j)=\sum_j T_{ij}Z^i$,  
such that there exists $N\in\Z$ for which
$T_{ij}=0$ whenever $|i-j|>N$. 

\medskip

The Lie algebra 
$\mathfrak{gl}_{\infty,\kappa}$ 
has a rich representation theory.
Let $\gl_\infty\subset \gl_{\infty,\kappa}$
be the Lie subalgebra of linear operators $T$ with finitely 
many non-zero matrix elements $T_{ij}$. 
Let  $W_\theta$ be the irreducible representation of $\mathfrak{gl}_\infty$
with the lowest weight
\begin{align*}
\theta=(\ldots,\theta_{-2},\theta_{-1},\theta_0,\theta_1,\theta_2,\ldots),\quad
E_{i,i}v_{\theta}=\theta_iv_{\theta},
\end{align*} 
 where $v_{\theta}$ is a lowest  
weight vector in $W_\theta$:
\begin{align*}
E_{i,j}v_\theta=0\text{ if $i>j$.}
\end{align*} 

If the sequence $\{\theta_i\}$ stabilizes as $i\rightarrow\pm\infty$,
then $W_\theta$ can be extended to the representation of 
${\mathfrak{gl}}_{\infty,\kappa}$.
Suppose that $\theta_i=\theta_-$ for $i\ll 0$, 
and $\theta_i=\theta_+$ for $i\gg 0$.
Then the central element $\kappa$ 
acts in $W_\theta$
by the scalar $\theta_--\theta_+$. 

Conjecturally all such $W_\theta$
can be deformed to the representations of $\mathcal E$. In this paper we 
confirm it in several special cases.

If 
$\theta=(\ldots,0,0,0,1,1,1,\ldots)$,
then $W_\theta$ is the well-known Fock representation given by semi-infinite wedges. In \cite{FFJMM1} the construction of the semi-infinite wedges was deformed and
as a result we get the Fock representation of  $\mathcal E$.
If the weight $\theta$ is 
anti-dominant, 
i.e.,  $\theta_i\in\Z$ and
$\theta_i-\theta_{i+1}\leq0$ for all $i\in \Z$,
the $\E$-modules corresponding to $W_{\theta}$ were also constructed in \cite{FFJMM1}, see also 
Section \ref{characters}.

Note that all these representations of $\mathcal E$ are described explicitly.
We have a natural basis and an explicit formula for the action on this basis.

\medskip

In the present paper we continue with the case 
$\theta(r)=(\ldots,0,0,0,r,r,r,\ldots)$,
where $r\in\C$ is generic. The character of $W_{\theta(r)}$
 in the principal grading is
given by the infinite product $\prod_{i=1}^\infty(1-q^i)^{-i}$. 
Incidentally, 
it coincides with the well-known 
Macmahon formula for the generating series of the plane partitions. 
Recall that a plane partition is a collection of 
non-negative integers $\{\mu^{(k)}_i\}_{i,k=1}^\infty$
satisfying $\mu^{(k)}_i\geq\mu^{(k+1)}_i$, $\mu^{(k)}_i\geq\mu^{(k)}_{i+1}$ for all $i,k$ and $\mu^{(k)}_i=0$ for $i+k$  large enough.

We construct a representation of $\mathcal E$ which is a deformation of $W_{(r)}$.
It depends on a complex parameter $K\not=0$, which is the value of a central element of $\E$
and is called the level of the representation.
It has an additional complex parameter $u\not=0$, which is related to an automorphism of $\E$.
Most importantly, it
has a distinguished basis labeled by the plane partitions and the action of $\E$ is explicit in this basis. We call this $\E$-module the Macmahon representation,
and denote it by
$\mathcal M(u,K)$.

Next, we observe that our construction has 
the following natural 
generalization.
Given three partitions $\al,\beta,\gamma$,  
we call a collection of numbers 
$\{\mu^{(k)}_j\}_{i,k=1}^\infty$,  
$\mu^{(k)}_j\in\Z_{\geq 0}\cup\{\infty\}$, 
a plane partition with the boundary condition
$\al,\beta,\gamma$,  if the following conditions are satisfied\\
(i)\quad
$\mu^{(k)}_i\geq\mu^{(k+1)}_i$, $\mu^{(k)}_i\geq\mu^{(k)}_{i+1}$ 
for all $i,k$,\\
(ii)\quad $\mu^{(k)}_i=\al_k$ for $i\gg 0$,\\
(iii)\quad $\mu^{(k)}_i=\gamma_i$ for $k\gg 0$,\\
(iv)\quad $\mu^{(k)}_i=\infty$ if and only if $i\le \beta_k$.\\ Let 
$\mathcal P[\al,\beta,\gamma]$ be the set of
plane partitions with the boundary condition $\al,\beta,\gamma$ .

The set $\mathcal P[\al,\beta,\gamma]$
appears in topological field theory as a fixed point set on Hilbert schemes
on toric $3$-dimensional Calabi-Yau manifolds 
\cite{ORV}.

We show that for generic values of $q_1,q_2, K$
the algebra $\mathcal E_{q_1,q_2,q_3}$ has an irreducible representation
$\mathcal M_{\al,\beta,\gamma}(u,K)$ depending on 
an additional
arbitrary complex parameter $u$ with a basis labeled by the set $\mathcal P[\al,\beta,\gamma]$ and 
give an explicit formula
for the action. 

Here the genericity assumption for $q_1,q_2$ means
$q_1^{i_1}q_2^{i_2}q_3^{i_3}\not=1$ unless 
$i_1=i_2=i_3$
and for $K$ means
$K\neq q_1^{i_1}q_2^{i_2}q_3^{i_3}$ for all integers $i_1,i_2,i_3$.
If $\al=\beta=\gamma=\emptyset$
we have $\mathcal M_{\al,\beta,\gamma}(u,K)=\mathcal M(u,K)$.

In the resonance case $K=q_1^{i_1}q_2^{i_2}q_3^{i_3}$
we show that the module $\mathcal M_{\al,\beta,\gamma}(u,K)$ is still well-defined, but becomes reducible. We describe singular vectors of $\mathcal M_{\al,\beta,\gamma}(u,K)$ and the irreducible quotient generated by the vector corresponding to the minimal
partition in $\mathcal P[\al,\beta,\gamma]$.
In the simplest case where
$\al=\beta=\gamma=\emptyset$ and $K=q_1q_2^{i_2}q_3^{i_3}$ $(i_2,i_3\geq1)$,
the irreducible quotient has a
basis labeled by 
plane partitions $\mu^{(k)}_i\in P(0,0,0)$ 
such that $\mu^{(i_2)}_{i_3}=0$.

\medskip

For the case of general $\al,\beta,\gamma$ the representation
$\mathcal M_{\al,\beta,\gamma}(u,K)$
does not have the 
limit $q_1\to 1$. But we show that if $\beta=\emptyset$,
it does and therefore it is a deformation of a 
$\gl_{\infty,\kappa}$-module.
Suppose further that for $n,c\in\Z_{\geq0}$
\begin{align*}
\gamma=(\underbrace{c,c,\dots,c}_n,0,0\dots),\
K=(q_2q_3)^n,
\end{align*}
then the irreducible quotient has the limit $q_1\rightarrow1$ and the limit
is an irreducible $\gl_{\infty,\kappa}$ module.
The lowest weight $\theta$ 
of this module is given in \eqref{GZ-hwt}. 

The basis of this irreducible quotient is labeled
by the set $P^n(\al,c)$ consisting
of all  $\{\mu_i^{(k)}\}_{i,k=1}^\infty\in P(\al,\emptyset,\gamma)$
such that $\mu_{n+1}^{(n+1)}=0$.
This basis leads us to find a Gelfand-Tsetlin type basis for the 
$\gl_{\infty,\kappa}$
module $W_{\theta}$, where $\theta$ is given by
\eqref{GZ-hwt}. 
The action is given
explicitly by Gelfand-Zetlin type formulas, see 
 \eqref{GZact1}, \eqref{GZact2}, \eqref{GZact3}, 
\eqref{GZact0}, \eqref{GZact00}.
We expect
that similar bases exist for all $\theta$, but we were unable to find them
in the literature (except for the standard case of the dominant weights).

Following \cite{KR2}, we give an explicit bosonic
construction of $W_{\theta}$. 
A version of the
Schur-Weyl-Howe duality established in \cite{KR2}
allows us to write bosonic 
character formulas for $W_\theta$ in the principal 
grading. 
Equivalently, our formula computes the generating function for the set $P^n(\al,c)$.

We do not have a recipe for how to compute the characters of the $\E$-modules 
$\mathcal M_{\al,\beta,\gamma}(u,K)$ where $K=q_1^{i_1}q_2^{i_2}q_3^{i_3}$
in general, but this problem seems to have a very intriguing structure.

For example, if $K=q_2^N$, then $\mathcal{E}$ has a big two-sided ideal. 
After factorization we get a smaller algebra $\mathcal{E}_{q_1,q_2,q_3,K}^{\mathrm{red}}$ which
can be identified with the elliptic $W$-algebra.
In particular, an appropriate limit of $\mathcal{E}_{q_1,q_2,q_3,K}^{\mathrm{red}}$ 
gives us the $W$-algebra for $\widehat{\mathfrak{gl}}_N$. The representation theory 
of the $W$-algebra for $\widehat{\mathfrak{gl}}_N$ is a well-known subject.
In particular, $W$-algebra has a class of representations appearing in the 
minimal models. In \cite{FFJMM2} we showed that there exist $\mathcal{E}$-modules which
have the same characters as the representations of minimal models of the $W$-algebra.  

In the case $K=q_2^mq_3^n$ the algebra $\mathcal{E}$ also has  a big two-sided ideal, and conjecturally the quotient in the appropriate limit gives us the
$W$-algebra for the superalgebra $\widehat{\mathfrak{gl}}(m,n)$.
This conjecture allows us to predict some character formulas, which can be checked by a computer for small values of parameters. We give some of such formulas at the end of the paper. 

\vskip10pt

The paper is organized as follows. 
In Section \ref{preliminaries} 
we recall and discuss some known facts
about algebra $\E$ and its representation. In Section \ref{Macmahon} 
we construct and study the Macmahon representations 
$\mathcal M_{\al,\beta,\gamma}(u,K)$. 
In Section \ref{GZbasis}, we study the $\gl_{\infty}$ 
limits of the Macmahon representations. In particular, we describe the Gelfand-Zetlin type basis for some  
$\gl_{\infty,\kappa}$-modules.
In Section \ref{characters} we construct 
the $\gl_{\infty}$-modules using Heisenberg algebra and compute their
characters by the Schur-Weyl-Howe duality of \cite{KR2}.
We finish with some conjectural character formulas.

\section{Preliminaries}\label{preliminaries} 

\subsection{Algebra $\mathcal E$}
Let $q_1,q_2,q_3\in\C$ be  
complex parameters satisfying the relation $q_1q_2q_3=1$. 
We assume that $q_1,q_2,q_3$ are not a root of unity. 
Let
\be
g(z,w)=(z-q_1w)(z-q_2w)(z-q_3w).
\en
The {\it quantum toroidal $\gl_1$} algebra
is an associative algebra $\E$ with generators
$e_i$, $f_i$, $i\in\Z$, 
$\psi^\pm_r$, $r\in\Z_{> 0}$ 
and invertible elements $\psi_0^\pm$,  
$C$,  
satisfying the following defining relations:
\begin{gather*}
\text{$C$: central},\label{Rel0}\\
\psi^\pm(z)\psi^\pm(w)=\psi^\pm(w)\psi^\pm(z),\label{Rel1}\\
g(Cz,w)g(Cw,z)\psi^+(z)\psi^-(w)=g(z,Cw)g(w,Cz)\psi^-(w)\psi^+(z)\,,\label{Rel2}\\
g(Cz, w)\psi^{+}(z)e(w)=-g( w,Cz)e(w)\psi^{+}(z),\\
g(z,w)\psi^{-}(z)e(w)=-g(w,z)e(w)\psi^{-}(z),\label{Rel3}\\
g(w,z)\psi^{+}(z)f(w)=-g(z,w)f(w)\psi^{+}(z),\\
g(w,Cz)\psi^{-}(z)f(w)=-g(Cz,w)f(w)\psi^{-}(z),\label{Rel4}\\
[e(z), f(w)]=\frac{1}{g(1,1)}(\delta(Cw/z)\psi^{+}(w)-\delta(Cz/w)\psi^{-}(z)),\label{Rel5}\\
g(z,w)e(z)e(w)=-g(w,z)e(w)e(z),\\
g(w,z)f(z)f(w)=-g(z,w)f(w)f(z),\label{Rel6}\\
[e_0,[e_1,e_{-1}]]=[f_0,[f_1,f_{-1}]]=0.
\label{Rel7}
\end{gather*}
Here $\delta(z)=\sum_{n\in\Z} z^n$ denotes the formal
delta function, and the generating series of the generators of $\E$ are given by
\begin{align*}
&e(z)=\sum_{i\in\Z} e_iz^{-i},
\quad  f(z)=\sum_{i\in\Z} f_iz^{-i},
\quad
\psi^\pm(z)=\sum_{\pm i\ge 0}\psi^{\pm}_iz^{-i}\,.
\end{align*}
Note that  $\E$ depends on the {\it unordered} set of parameters $\{q_1,q_2,q_3\}$, 
as all $q_i$ enter the relations symmetrically through the function $g(z,w)$.

Algebra $\E$ has been introduced and studied by Miki \cite{Miki07}
under the name ``$(q,\gamma)$ analog of the $W_{1+\infty}$ algebra''. 
(To be precise, in \cite{Miki07} an additional relation 
$\psi^+_0\psi^-_0=1$ is imposed, and 
$\E$ is a one-dimensional split central extension of that of \cite{Miki07}.)

Consider the associative $\C$ algebra with generators $Z^{\pm1}$, $D^{\pm1}$ with the relation $DZ=qZD$. 
Let $\mathfrak d_q$ be the same algebra viewed as a Lie algebra by $[a,b]=ab-ba$. 
Then $\mathfrak d_q$ has a two-dimensional central extension \cite{KR}
$\mathfrak d_{q,c_1,c_2}=\mathfrak d_q\oplus \C c_1\oplus \C c_2$,  
where $c_1,c_2$ are central elements and the commutator is given by
\begin{align*}
[Z^{i_1}D^{j_1},Z^{i_2}D^{j_2}]= 
(q^{j_1i_2}-q^{j_2i_1})
Z^{i_1+i_2}D^{j_1+j_2}
+\delta_{i_1+i_2,0}
\delta_{j_1+j_2,0}q^{-i_1j_1}(i_1c_1+j_1c_2).
\end{align*}
The element $Z^0D^0=1$ is (split) central in $\mathfrak{d}_{q,c_1,c_2}$. 
The quantum toroidal $\mathfrak{gl}_1$ algebra $\E$
is a quantization of the universal enveloping algebra $U\mathfrak d_{q,c_1,c_2}$ 
where $q_1$ is a parameter of the quantization and 
$q_2=q^2$. 
Algebra $\E$ has three central elements, $C$  
 and $\psi^\pm_0$. 
Among the latter only the ratio $(\psi^+_0)^{-1} \psi^-_0$ is essential.
We say that an $\E$-module $V$ has {\it level $(x,y)\in \C^2$} if 
$C^{2}$ acts by $x$ and $(\psi^+_0)^{-1}\psi_0^-$ acts by  $y$. 

In what follows, we shall always consider 
representations of $\E$ on which $C$ acts as identity.  
In other words we study representations of the quotient algebra $\E/\langle C-1\rangle$,  
where the defining relations simplify as follows.
\begin{gather}
\label{rel1}
\psi^\epsilon(z)\psi^{\epsilon'}(w)=\psi^{\epsilon'}(w)\psi^\epsilon(z)\quad
(\epsilon,\epsilon'\in\{+,-\})\,,\\
\label{rel2}
g(z,w)\psi^{\pm}(z)e(w)=-g(w,z)e(w)\psi^{\pm}(z),\\
g(w,z)\psi^{\pm}(z)f(w)=-g(z,w)f(w)\psi^{\pm}(z),\\
\label{rel3}
[e(z), f(w)]=\frac{\delta(z/w)}{g(1,1)}(\psi^{+}(w)-\psi^{-}(z)),\\
\label{rel4}
g(z,w)e(z)e(w)=-g(w,z)e(w)e(z),\\
g(w,z)f(z)f(w)=-g(z,w)f(w)f(z),\\
\label{rel5}
[e_0,[e_1,e_{-1}]]=[f_0,[f_1,f_{-1}]]=0.
\end{gather}
In the quotient algebra, the subalgebra generated by $\psi^\pm_{\pm i}$, $i\in\Z_{\ge0}$, is commutative. 
We call an $\E$-module $V$ {\it tame} if $\psi^\pm_{\pm i}$, 
$i\in\Z_{\ge0}$, act by 
diagonalizable operators with simple joint spectrum.

Algebra $\E$ is $\Z^2$-graded with the assignment
\begin{align*}
\mathrm{deg}\,e_i=(1,i)\,,
\quad
\mathrm{deg}\,f_i=(-1,i)\,,
\quad
\mathrm{deg}\,\psi^\pm_i=(0,i)\,,
\quad
\mathrm{deg}\,C=(0,0)\,.
\end{align*}
Let $\E_{(j,k)}$ denote the homogeneous component 
of $\E$ of degree $(j,k)\in\Z^2$.
We say that an $\E$ module $V=\oplus_{n\in\Z}V_n$ is {\it $\Z$-graded} if 
$\E_j V_n\subset V_{n+j}$ where $\E_{j}=\sum_{k\in\Z}\E_{(j,k)}$.
We call it {\it quasifinite} if $\mathrm{dim}\,V_n<\infty$  for all $n\in\Z$. 
Let $\phi^\pm(z)\in \C[[z^{\mp1}]]$ be formal power series in $z^{\mp1}$ with non-vanishing constant term. We say that a $\Z$-graded $\E$ module $V$ is a {\it lowest weight module} 
with {\it lowest weight} $(\phi^+(z),\phi^-(z))$ 
if it is generated by a non-zero vector $v$ such that
\begin{align*}
f(z)v=0,\quad 
\psi^\pm(z)v=\phi^\pm(z)v,
\quad C v=v\,.
\end{align*}

The following result due to Miki is an analog of the classification 
theorem for finite dimensional modules of 
quantum affine algebras. 
\begin{thm}[\cite{Miki07}]\label{highest weight thm}
Up to isomorphisms,  
an irreducible lowest weight module $V$ is 
uniquely determined by its lowest weight. 
It is quasifinite if and only if
there exists a rational function $R(z)$, which is regular and non-zero
at $z=0,\infty$,
such that 
$\phi^\pm(z)$ is the expansion of $R(z)$ at $z^{\pm1}=\infty$.
\end{thm}

{\it Remark.}\quad 
In \cite{Miki07}, highest weight modules are considered.
Though this is purely a matter of convention,
in this paper we shall deal with lowest weight modules for historical reasons. 
We note that we have called the same object `highest weight modules' 
in \cite{FFJMM1}, \cite{FFJMM2}. 

\medskip

Algebra $\E$ has the formal comultiplication
\begin{gather}\label{tre}
\triangle e(z)= e(z)\T 1 + \psi^-(z)\T e(z),\\
\label{trf}
\triangle f(z)= f(z)\T \psi^+(z) + 1\T f(z),\\
\label{trpsi}
\triangle \psi^\pm(z)= \psi^\pm(z)\T \psi^\pm(z).
\end{gather}
These formulas do not define a comultiplication in the usual sense
since the right hand sides contain infinite sums.  
In \cite{Miki07}, it is shown that the 
twisted coproduct by a certain automorphism
is well defined on tensor products of 
a class of modules (called restricted modules).
In this paper, we shall take a slightly different approach  
and use the original coproduct given by \eqref{tre}, \eqref{trf}, \eqref{trpsi}
when it makes sense. 
The arguments for justification can be found e.g. in the proof of Proposition 3.1 in \cite{FFJMM1}.

\subsection{Fock modules}
Let $u\in\C$. Let $V(u)=V_1(u)$ be a complex vector space 
spanned by basis $[u]_i$, $i\in\Z$. Then the formulas
\begin{align}
&(1-q_1)e(z)[u]_i=\delta(q_1^iu/z)[u]_{i+1}\,,\notag \\
&-(1-q_1^{-1})f(z)[u]_i=\delta(q_1^{i-1}u/z)[u]_{i-1}\,,\notag \\
&\psi^+(z)[u]_i=\frac{(1-q_1^iq_3u/z)(1-q_1^iq_2u/z)}{(1-q_1^iu/z)(1-q_1^{i-1}u/z)}[u]_i,
\label{PSIPLUSACTION}\\
&\psi^-(z)[u]_i=\frac{(1-q_1^{-i}q_3^{-1}z/u)(1-q_1^{-i}q_2^{-1}z/u)}{(1-q_1^{-i}z/u)(1-q_1^{-i+1}z/u)}[u]_i,
\label{PSIMINUSACTION}
\end{align}
define a structure of an irreducible tame quasifinite $\E$-module on $V(u)$ 
of level $(1,1)$.
We call the $\E$-module $V(u)$ the 
{\it vector representation}. 
The vector representation is not a lowest
weight representation, it is the counterpart of 
the $\mathfrak d_q$ 
module $\C[Z,Z^{-1}]$.

Note that
$q_1$ plays a special role in the definition of $V(u)$ while $q_2$ and
$q_3$ participate symmetrically. Therefore there are two other vector
representations obtained from $V(u)$ by switching roles of $q_i$. 

We set
\begin{align*}
\psi_i(z)=\psi(q_1^iz), \qquad \psi(z)=\frac{(1-q_3z)(1-q_2z)}{(1-z)(1-q_2q_3z)}.
\end{align*}
By \Ref{PSIPLUSACTION}, \Ref{PSIMINUSACTION}, we have $\psi^\pm(z)[u]_i=\psi_i(u/z)[u]_i$.

The Fock representation 
$\mathcal F(u)=\mathcal F_2(u)=\oplus_\la\C|\la\rangle$
is constructed in the infinite tensor product of
the vector representations (see \cite{FFJMM1}):
\begin{align}
&\mathcal F(u)\subset V(u)\otimes V(uq_2^{-1})\otimes V(uq_2^{-2})\otimes\cdots,\nn\\
&|\la\rangle=[u]_{\la_1}\otimes[uq_2^{-1}]_{\la_2-1}\otimes[uq_2^{-2}]_{\la_3-2}\otimes\cdots.
\label{INFINITETENSOR}
\end{align}
Here $\la=(\la_1,\la_2,\ldots)$ is a partition: $\la_1\geq\la_2\geq\cdots\geq0$, $\la_j\in\Z_{\geq0}$;
$\la_N=0$ for a large $N$. We denote the corresponding Young diagram by $Y_\la$.

In this tensor product we have two problems. First, we should avoid poles in the substitution.
Let us examine the factor in the action of $e(z)$.

We prepare some notations.
Let $\la\pm1_i$ denote the partition $\mu$ such that $\mu_j=\la_j$ if $j\neq i$ and
$\mu_i=\la_i\pm1$.
We say $(i,j)$ is a concave (resp., convex) corner of $\la$ if and only if
\begin{align*}
\la_i=j-1<\la_{i-1}\quad(\text{resp., }\la_i=j>\la_{i+1}).
\end{align*}
Denote by $CC(Y_\la)$ (resp., $CV(Y_\la)$) the set of concave (resp.,
convex) corners of $\la$.

From the comultiplication rule we have
\begin{align}
&e(z)|\la\rangle=\sum_{i=1}^\infty\frac{\psi_{\la,i}}{1-q_1}\delta(q_1^{\la_i}q_3^{i-1}u/z)|\la+1_i\rangle,\notag \\ 
&\psi_{\la,i}=\prod_{k=1}^{i-1}
\frac{(1-q_1^{\la_k-\la_i}q_3^{k-i+1})(1-q_1^{\la_k-\la_i-1}q_3^{k-i-1})}
{(1-q_1^{\la_k-\la_i}q_3^{k-i})(1-q_1^{\la_k-\la_i-1}q_3^{k-i})}.\label{PSILA}
\end{align}
Note that $\psi_{\la,i}$ has no pole. It has a zero when $\la_{i-1}=\la_i$. This zero prohibits a term $|\mu\rangle$
with $\mu=\la+1_i$ which breaks the condition $\mu_{i-1}\geq\mu_i$ from appearing in the right hand side.
Thus the above sum reduces to a finite sum:
\begin{align*}
e(z)|\la\rangle=\sum_{(i,j)\in CC(\la)}\frac{\psi_{\la,i}}{1-q_1}\delta(q_1^{j-1}q_3^{i-1}u/z)|\la+1_i\rangle.
\end{align*}

Second, when we deal with
the semi-infinite tensor product we have to give
a meaning to the infinite product which appears in the action of $\psi^\pm(z)$ and $f(z)$.
Let us give a meaning to the infinite product which appears in the action of $\psi^\pm(z)$:
\begin{align*}
&\psi^\pm(z)|\la\rangle=\psi_\la(u/z)|\la\rangle,\quad
\psi_\la(u/z)=\prod_{i=1}^\infty\psi_{\la_i-i+1}(uq_2^{-i+1}/z).
\end{align*}
The product can be written as
\begin{align}
\psi_\la(u/z)=\frac{1-q_1^{\la_1-1}q_3^{-1}u/z}{1-q_1^{\la_1}u/z}
\prod_{j=1}^\infty
\frac{(1-q_1^{\la_j}q_3^ju/z)(1-q_1^{\la_{j+1}-1}q_3^{j-1}u/z)}
{(1-q_1^{\la_{j+1}}q_3^ju/z)(1-q_1^{\la_j-1}q_3^{j-1}u/z)},\label{PSI}
\end{align}
which is convergent because of the boundary condition $\la_N=0$ for large $N$.
We remark that the convergence is valid, in general, if $\lim_{i\rightarrow\infty}\la_i$ exists.
We will use \eqref{PSI} under that condition later.
This formula implies that the level 
of $\mathcal{F}(u)$
is $(1,q_2)$.
For the vacuum $|\emptyset\rangle$, 
i.e., the empty Young diagram, we have
\begin{align}
&\psi_\emptyset(u/z)=\frac{1-q_2u/z}{1-u/z}.\label{HW}
\end{align}
This is the lowest weight of $\mathcal{F}(u)$.
The general formula \eqref{PSI} for $|\la\rangle$ can be understood as starting from the 
lowest
weight \eqref{HW}
for the vacuum, and multiplying the contribution from each box of $Y_\la$.
Namely, set
\begin{align*}
\psi_{i,j}(u/z)&=\frac{\psi_{j-i+1}(uq_2^{-i+1}/z)}{\psi_{j-i}(uq_2^{-i+1}/z)}\\
&=\frac{(1-q_1^jq_3^iu/z)(1-q_1^{j-2}q_3^{i-1}u/z)(1-q_1^{j-1}q_3^{i-2}u/z)}
{(1-q_1^{j-1}q_3^iu/z)(1-q_1^jq_3^{i-1}u/z)(1-q_1^{j-2}q_3^{i-2}u/z)}.
\end{align*}
The rational function $\psi_\la(u/z)$ can be determined recursively by
\begin{align}
&\psi_{\la}(u/z)=\psi_{i,\la_i}(u/z)\psi_{\la-1_i}(u/z).\label{REC}
\end{align}
This formula immediately follows from \eqref{PSIPLUSACTION},\eqref{PSIMINUSACTION}
and the comultiplication rule. It says that
the contribution from the box $(i,j)$ is $\psi_{i,j}(u/z)$.
Using \eqref{REC}, it is easy to see that
\begin{align}
\psi_\la(u/z)=\prod_{(i,j)\in CC(\la)}\frac{1-q_1^{j-2}q_3^{i-2}u/z}{1-q_1^{j-1}q_3^{i-1}u/z}
\prod_{(i,j)\in CV(\la)}\frac{1-q_1^jq_3^iu/z}{1-q_1^{j-1}q_3^{i-1}u/z}\,.
\label{NORMALIZATION}
\end{align}
From this, one can see that the representation is tame.

\begin{figure}
\begin{picture}(150,150)(40,-30)

\put(0,60){\line(1,0){120}}
\put(0,40){\line(1,0){120}}
\put(0,20){\line(1,0){20}}

\put(120,80){\line(0,-1){40}}
\put(100,80){\line(0,-1){40}}
\put(80,80){\line(0,-1){40}}
\put(60,80){\line(0,-1){40}}
\put(40,80){\line(0,-1){40}}
\put(20,80){\line(0,-1){60}}

\put(-15,85){$\la$}
\put(-20,-30)
{$CC(\la)=\{(1,7),(3,2),(4,1)\},\ CV(\la)=\{(2,6),(3,1)\}$}
\put(160,80){$j,q_1$}
\put(-10,-10){$i,q_3$}

\thicklines
\put(0,80){\vector(1,0){150}}
\put(0,80){\vector(0,-1){80}}

\end{picture}
\caption{Partition}
\end{figure}
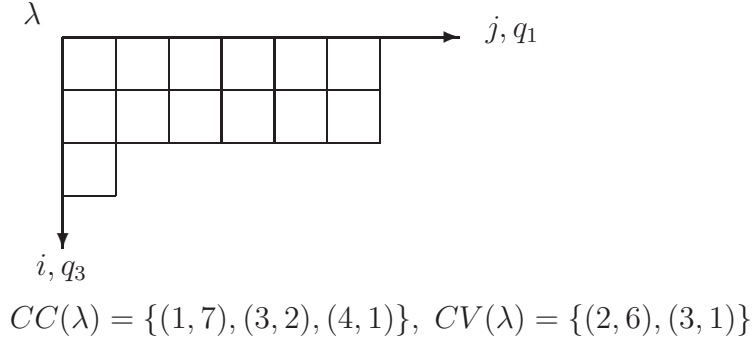

The formula for the action of $f(z)$ is obtained similarly:
\begin{align}
&f(z)|\la\rangle=\sum_{i=1}^\infty\frac{q_1\psi'_{\la,i}}{1-q_1}\delta(q_1^{\la_i-1}q_3^{i-1}u/z)|\la-1_i\rangle,\notag \\
&\psi'_{\la,i}=\frac{1-q_1^{\la_{i+1}-\la_i}}{1-q_1^{\la_{i+1}-\la_i+1}q_3}
\prod_{k=i+1}^\infty\frac{(1-q_1^{\la_k-\la_i+1}q_3^{k-i+1})(1-q_1^{\la_{k+1}-\la_i}q_3^{k-i})}
{(1-q_1^{\la_{k+1}-\la_i+1}q_3^{k-i+1})(1-q_1^{\la_k-\la_i}q_3^{k-i})}.\label{PSIPRIME}
\end{align}
Again, $\psi'_{\la,i}$ has no pole, and the zero at $\la_i=\la_{i+1}$ prohibits the appearance of terms
$|\mu\rangle$ which breaks the condition $\mu_i\geq\mu_{i+1}$. Thus, the action of $f(z)$ reads as
\begin{align}
&f(z)|\la\rangle=\sum_{(i,j)\in CV(\la)}\frac{q_1\psi'_{\la,i}}{1-q_1}\delta(q_1^{j-1}q_3^{i-1}u/z)|\la-1_i\rangle.
\label{FACTION}
\end{align}

If
we exchange $q_1$ with $q_3$ the representation $\mathcal F(u)$ changes. Let us denote it
by ${\mathcal F}'(u)$. This representation is realized 
inside
the semi-infinite tensor product $V_3(u)\otimes V_3(uq_2^{-1})\otimes V_3(uq_2^{-2})\otimes\cdots$.
Since  ${\mathcal F}'(u)$ has the same lowest weight 
\eqref{HW} as that of $\mathcal F(u)$, these
two modules are isomorphic by Theorem \ref{highest weight thm}.
We can construct the isomorphism explicitly.
Look at the action of $\psi^\pm(z)$ \eqref{NORMALIZATION}.
If we exchange $(i,j)\leftrightarrow(j,i)$ and $q_1\leftrightarrow q_3$ simultaneously,
the factors are invariant.
Therefore, as $\psi_\pm(z)$ modules, for arbitrary nonzero constants $c_\la$
the mapping $|\la\rangle\mapsto c_\la|\la'\rangle$,  
with $\la'$ being the transpose of $\la$,
is an intertwiner.
Since the representations are tame, this is the only way of intertwining these two $\E$-modules.
Theorem \ref{highest weight thm} shows the existence of the set of constants $c_\la$.

Now let us discuss the relation \eqref{rel3}.
On the subspace of $V(u)\otimes V(uq_2^{-1})\otimes\cdots\otimes V(uq_2^{-N+1})$
that is spanned by the vector
\begin{align}
&|\la\rangle^{(N)}=[u]_{\la_1}\otimes[uq_2^{-1}]_{\la_2-1}
\otimes\cdots\otimes[uq_2^{-N+1}]_{\la_N-N+1}\ ,\label{FINITETENSOR}
\end{align}
where $\la_1\geq\la_2\geq\cdots\geq\la_N$, the relation \eqref{rel4} is valid.
The left hand side is a sum of products of delta functions of the form
$c(q_1,q_3)\delta(z/w)\delta(q_1^aq_3^bu/z)$. The right hand side is a difference of 
two series $\psi^\pm(z)$: one is obtained from the rational function with simple poles,
\begin{align*}
&\psi^{(N)}_\la(u/z)=\prod_{j=1}^N\psi_{\la_j-j+1}(q_2^{-j+1}u/z),
\end{align*}
by expanding it in $z^{-1}$,
and the other from the same rational function by expanding it in $z$.
The difference is a sum of delta functions. They can be computed from the position of the poles
of $\psi^{(N)}_\la(u/z)$ and their residues.

Now, consider the semi-infinite action on \eqref{INFINITETENSOR},
and compare it with the action on \eqref{FINITETENSOR}. If $N$ is large enough so that
$\la_N=0$, we identify \eqref{INFINITETENSOR} with \eqref{FINITETENSOR}.
Then the action of $e(z)$ is the same.
The actions of $\psi^\pm(z)$ are slightly different:
\begin{align}
\psi^{(N)}_\la(u/z)=\psi_\la(u/z)\frac{1-q_3^Nu/z}{1-q_1^{-1}q_3^{N-1}u/z}.\label{MODIFICATION}
\end{align}
The factor $\frac{1-q_3^Nu/z}{1-q_1^{-1}q_3^{N-1}u/z}$
is dropped in the action of $\psi^\pm(z)$ on $\mathcal F(u)$.
We also drop the same factor from the action of $f(z)$.
This explains the factor $\psi'_{\la,i}$ in \eqref{FACTION}: for large $N$, we have 
\begin{align}
\psi'_{\la,i}=\prod_{k=i+1}^N\psi_{\la_k-k+1}(q_2^{-k+1}u/z)
\times\frac{1-q_3^Nq_2u/z}{1-q_3^Nu/z}
\Bigg|_{u/z\rightarrow q_1^{-\la_i+1}q_3^{-i+1}}.\label{PSILAP}
\end{align}

Since the tensor product \eqref{INFINITETENSOR} does not have the $(N+1)$st component,
the equality \eqref{rel4} for \eqref{FINITETENSOR} contains an extra term with
the delta function $\delta(q_1^{-1}q_3^{N-1}u/z)$. On the other hand, in the action on $\mathcal F(u)$,
this term is killed by the zero of $\psi'_{\la,N}$ as discussed before;
going to \eqref{INFINITETENSOR}
this term is dropped in both sides of \eqref{rel4}. The effect of the modification \eqref{MODIFICATION}
is the same on each delta function term because it is the 
multiplication by
the same factor.
Thus the equality \eqref{rel4} is valid on $\mathcal F(u)$.

The value of $\psi^-_0$ has been changed by the multiplication because the value of this factor at $z=0$ is $q_2$.
The modification produces the non-trivial level $(1,q_2)$ for the representation $\mathcal F(u)$.

\section{Macmahon modules}\label{Macmahon} 
 \subsection{Vacuum Macmahon modules}
Let us construct a level $(1,K)$ representation
\begin{align}
\mathcal M(u,K)\subset\mathcal F(u)\otimes\mathcal F(uq_2)\otimes\mathcal F(uq_2^2)\otimes\cdots
\label{INFTEN0}
\end{align}
with basis
\begin{align*}
\mathcal M(u,K)=\bigoplus_{\lab}\C|\lab\rangle,\
\lab=(\la^{(1)},\la^{(2)},\la^{(3)},\cdots), 
\end{align*}
where $\lab$ is a plane partition, i.e., each $\la^{(k)}=(\la^{(k)}_1,\la^{(k)}_2,\cdots,0,0,\cdots)$
is a partition and
\begin{align}
\la^{(k)}_i\geq\la^{(k+1)}_i\label{PP}
\end{align}
is satisfied. We require $\la^{(N)}=\emptyset$ for large $N$.
In particular, we set
\begin{align*}
\boldsymbol\emptyset
=(\emptyset,\emptyset,\emptyset,\cdots).
\end{align*}
With each $\lab$ we associate a subset $Y_\lab$ of $\bigl(\Z_{\geq1}\bigr)^3$ such that
$(i,j,k)\in Y_\lab$ if and only if $j\leq\la^{(k)}_i$. This is a finite set.

We call the representation $\mathcal M(u,K)$ the vacuum Macmahon representation.

In \cite{FFJMM2}, Theorem 3.4, the action of $\mathcal E$ was defined on the subspace
$\mathcal M^{(n)}_{\mathbf{a,b}}$
($\mathbf a=(a_1,\ldots,a_{n-1}),\mathbf b=(b_1,\ldots,b_{n-1})$)
of $\mathcal F(u_1)\otimes\mathcal F(u_2)\otimes\cdots\otimes\mathcal F(u_n)$ where
\begin{align}
u_{i+1}=u_iq_1^{-a_i}q_2q_3^{-b_i}\label{UU}
\end{align}
spanned by the vectors
\begin{align*}
|\lab^{(n)}\rangle=|\la^{(1)}\rangle\otimes\cdots\otimes|\la^{(n)}\rangle
\end{align*}
satisfying
\begin{align}
\la^{(k)}_i+a_k\geq\la^{(k+1)}_{i+b_k}.\label{ELEVATION}
\end{align}

The tensor product \eqref{INFTEN0} with the restriction \eqref{PP}
is the limit $n\to\infty$
of this construction in the case $\mathbf{a}=\mathbf{b}=\mathbf{0}$.
From the discussion in the previous subsection, the method for constructing the action
on the infinite tensor product based on the results in \cite{FFJMM2} is clear.
However, we must be careful on the definition of $\psi^\pm(z)$ since the level of the representation
$\mathcal M^{(n)}_{\mathbf{a,b}}$ is $(1,q_2^n)$, and the simple-minded limit $n\rightarrow\infty$ is not defined.
In fact, this is not a defect but here is a room for introducing an arbitrary parameter for the level.

We define the action of $\psi^\pm(z)$ by
\begin{align}
&\psi^\pm(z)|\lab\rangle=\psi_\lab(u/z)|\lab\rangle,\label{Psi1}\\
&\psi_\lab(u/z)=
\psi_{\boldsymbol\emptyset}(u/z)
\prod_{(i,j,k)\in{\textstyle Y_\lab}}\psi_{i,j,k}(u/z),\notag\\ 
&\psi_{\boldsymbol\emptyset}
(u/z)=\frac{1-Ku/z}{1-u/z},\notag \\ 
&\psi_{i,j,k}(u/z)
=\frac{(1-q_1^jq_2^{k-1}q_3^iu/z)(1-q_1^{j-1}q_2^kq_3^{i}u/z)(1-q_1^{j}q_2^kq_3^{i-1}u/z)}
{(1-q_1^{j-1}q_2^kq_3^{i-1}u/z)(1-q_1^{j-1}q_2^{k-1}q_3^iu/z)(1-q_1^jq_2^{k-1}q_3^{i-1}u/z)}.\label{Psi4}
\end{align}
Here $K$ is an arbitrary nonzero
parameter. The level of representation is $(1,K)$. It is easy to see that
the action of $\psi^\pm(z)$ is tame. In fact, the partition $\la^{(1)}$ can be read from
$\psi_\lab(u/z)$ by identifying its concave and convex corners recursively:
the rightmost concave corner $(i_1+1,j_1)=(1,\la^{(1)}_1+1)$ can be identified by
the pole coming from the factor $1-q_1^{\la^{(1)}_1}u/z=1-q_1^{j_1-1}q_3^{i_1}u/z$. Among the factors
of the form $(1-q_1^xu/z)$ in the denominator of $\psi_\lab(u/z)$, the one with
$x=\la^{(1)}_1$ is the largest in $x$.
Next, the rightmost convex corner $(i_2,j_1-1)$ can be identified
by the zero at $(1-q_1^{j_1-1}q_3^{i_2}u/z)$. 
Among the factors of the form $(1-q_1^{j_1-1}q_3^xu/z)$
in the numerator of $\psi_\lab(u/z)$,
the one with $x=i_2$ is the largest in $x$.
Similarly, one can identify the concave corner $(i_2+1,j_2)$, then the convex corner $(i_3,j_2-1)$,
etc., from the factors in $\psi_\lab(u/z)$. After identifying $\la^{(1)}$,
we divide $\psi^{(1)}(u/z)=\psi_\lab(u/z)$ by the factors corresponding to $\la^{(1)}$, and obtain new $\psi^{(2)}(u/z)$.
Then, one can determine $\la^{(2)}$ by the same procedure using this $\psi^{(2)}(u/z)$.
Continuing in this way, we can completely determine $\lab$ from $\psi_\lab(u/z)$.

The rational function $\psi_\lab(u/z)$ can be determined recursively.
Denote $\mub$ such that
$\mu^{(m)}=\la^{(m)}$ if $m\not=k$, and $\mu^{(k)}=\la^{(k)}\pm1_i$ by $\lab\pm1^{(k)}_i$.
Then we have
\begin{align*}
\psi_\lab(u/z)=\psi_{i,\la^{(k)}_i,k}(u/z)\psi_{\lab-1^{(k)}_i}(u/z).
\end{align*}

Let us compare $\psi_\lab(u/z)$ with
\begin{align}
\psi^{(k)}_\lab(u/z)=\prod_{m=1}^k\psi_{\la^{(m)}}(u_m/z),\quad u_m=uq_2^{m-1}.\label{PSIK}
\end{align}
For $N>\hskip-3pt>1$, we have
\begin{align}
\psi_\lab(u/z)=\psi^{(N)}_\lab(u/z)\frac{1-Ku/z}{1-q_2^Nu/z}.\label{LEVELK}
\end{align}
This is because for large $N$ we have the same recursion
\begin{align*}
\psi^{(N)}_\lab(u/z)=\psi_{i,\la^{(k)}_i,k}(u/z)\psi^{(N)}_{\lab-1^{(k)}_i}(u/z).
\end{align*}
Note that the structure of poles is
the same for $\psi^{(N)}_\lab(u/z)$ and $\psi_\lab(u/z)$
because the former (for large $N$) has a zero at $1-q_2^Nu/z=0$ and the latter at $1-Ku/z=0$.
This is important in the derivation of \eqref{rel4}. Namely, the position of delta functions appearing
in the right hand side of the equality
does not change by changing the rational function from $\psi^{(N)}_\lab(u/z)$ to $\psi_\lab(u/z)$.
It is also invariant in the left hand side because for large $N$ we have $\la^{(N)}=\emptyset$ and
$|\la^{(N)}\rangle=|\emptyset\rangle$ is the lowest
weight vector, which is killed by the action of $f(z)$.
Thus, we can establish the existence of
the representation on $\mathcal M(u,K)$ with the level $(1,K)$ and the lowest
weight
$\frac{1-Ku/z}{1-u/z}$.

For completeness we give the action of $e(z),f(z)$ on $\mathcal M(u,K)$.

The action of $e(z)$ on $|\lab\rangle$ is defined by
\begin{align}
e(z)|\lab\rangle=
\sum_{k=1}^\infty\psi^{(k-1)}_\lab(u/z)\sum_{i=1}^\infty
\psi_{\la^{(k)},i}\frac1{1-q_1}\delta(q_1^{\la^{(k)}_i}q_2^{k-1}q_3^{i-1}u/z)
|\lab+1^{(k)}_i\rangle.\label{EACT}
\end{align}
From \cite{FFJMM2} follows that for the finite tensor product
the delta function does not pick up poles of $\psi^{(k-1)}_\lab(u/z)$, and does
pick up a zero if and only
if $\mub=\lab+1^{(k)}_i$ breaks the condition $\mu^{(k-1)}_i\geq\mu^{(k)}_i$.
Here we give a simple proof of these statements using \eqref{NORMALIZATION}.

Set $j=\la^{(k)}_i+1$. Suppose that
\begin{align*}
(i,j)\in CC(\la^{(k)}).
\end{align*}
From \eqref{NORMALIZATION} we see that
the function $\psi^{(k-1)}_\lab(u/z)$ has a pole at $q_1^{j-1}q_2^{k-1}q_3^{i-1}u/z=1$
only if for some $m\leq k-1$, there exists a box $(\bar i,\bar j)$ such that
\begin{align*}
(\bar i,\bar j)\in CC(\la^{(m)})\sqcup CV(\la^{(m)}) 
\end{align*}
and
\begin{align*}
q_1^{j-1}q_2^{k-1}q_3^{i-1}=q_1^{\bar j-1}q_2^{m-1}q_3^{\bar i-1}.
\end{align*}
The latter implies
\begin{align*}
\bar i=i+m-k<i,\ \bar j=j+m-k<j.
\end{align*}
This is a contradiction with $\la^{(m)}_l\geq\la^{(k)}_l$. We have shown the 
statement about the 
poles.

Let us show the statement about the zeros.  
A zero occurs only if either
\begin{align*}
(\bar i,\bar j)\in CC(\la^{(m)})\quad \text{ and }\quad
q_1^{j-1}q_2^{k-1}q_3^{i-1}=q_1^{\bar j-2}q_2^{m-1}q_3^{\bar i-2}
\end{align*}
or
\begin{align*}
(\bar i,\bar j)\in CV(\la^{(m)})\quad \text{ and }\quad
q_1^{j-1}q_2^{k-1}q_3^{i-1}=q_1^{\bar j}q_2^{m-1}q_3^{\bar i}.
\end{align*}
The former case really occurs when
the condition $\mu^{(k-1)}_i\geq\mu^{(k)}_i$ is broken, while the latter  leads to a contradiction.

A box $(i,j,k)$ is called a concave (resp., convex) corner of $Y_\lab$ if
\begin{align*}
&(i,j,k)\not\in Y_\lab\ \text{ (resp., $(i,j,k)\in Y_\lab$)}
\end{align*}
and
\begin{align*}
&(i-1,j,k),(i,j-1,k),(i,j,k-1)\in Y_\lab\\[7pt]
&\qquad\text{ (resp., $(i+1,j,k),(i,j+1,k),(i,j,k+1)\not\in Y_\lab$)}. 
\end{align*}
We denote by $CC(Y_\lab)$ (resp., $CV(Y_\lab)$) the set of concave (resp., convex) corners of $\lab$.
They are finite sets.
The action of $e(z)$ adds a box at each concave corner (see \eqref{PSILA}, \eqref{PSIK}):
\begin{align}
&e(z)|\lab\rangle=\sum_{(i,j,k)\in
{\scriptstyle CC(Y_\lab)}}\psi_{\lab,i,j,k}
\psi_{\la^{(k)},i}\frac1{1-q_1}\delta(q_1^jq_2^kq_3^iu/z)|\lab+1^{(k)}_i\rangle,\label{EACTION}\\
&\psi_{\lab,i,j,k}=\psi^{(k-1)}_\lab(q_1^{-j}q_2^{-k}q_3^{-i}).\label{PSIIJK}
\end{align}
Similarly, we have the formula for the action of $f(z)$ (see \eqref{PSILAP}).
\begin{align*}
&f(z)|\lab\rangle=\sum_{(i,j,k)\in
{\scriptstyle CV(Y_\lab)}}
\psi'_{\lab,i,j,k}
\psi'_{\la^{(k)},i}\frac{q_1}{1-q_1}\delta(q_1^jq_2^kq_3^iu/z)|\lab-1^{(k)}_i\rangle,\\
&\psi'_{\lab,i,j,k}=\psi'_\lab\hskip-3pt{^{(k+1)}}(q_1^{-j}q_2^{-k}q_3^{-i}),\\
&\psi'_\lab\hskip-3pt{^{(k)}}(u/z)=\lim_{N\rightarrow\infty}
\prod_{m=k}^N\psi_{\la^{(m)}}(q_2^{m-1}u/z)
\times
\frac{1-Ku/z}{1-q_2^Nu/z}.
\end{align*}
As we discussed $\psi'_{\la^{(k)},i}$ (see \eqref{PSIPRIME}) has no pole,
and it has a zero if and only if $\mub=\lab-1^{(k)}_i$ breaks the condition for the plane partitions.
The discussion for poles and zeros of $\psi'_{\lab,i,j,k}$ is exactly the same as $\psi_{\lab,i,j,k}$ for $e(z)$.
\subsection{Macmahon modules with non-trivial boundary conditions}
In this subsection we generalize the Macmahon representation to the case where
the plane partitions have non-trivial boundary conditions.
We repeat the semi-infinite tensor product construction.
We remove the restriction $\mathbf{a}=\mathbf{b}=\mathbf{0}$
in \eqref{ELEVATION}, and also remove the condition $\la^{(N)}=\emptyset$ for large $N$.

It is convenient to use another notation. 
Consider a set of three partitions $\al=(\al_1,\al_2,\ldots,0,\ldots)$, $\beta=(\beta_1,\beta_2,\ldots,0,\ldots)$,
$\gamma=(\gamma_1,\gamma_2,\ldots,0,\ldots)$. We call a sequence $\mu=(\mu_1,\mu_2,\mu_3,\ldots)$ where
$\mu_i\in\Z_{\geq0}\sqcup\{\infty\}$ a generalized partition if and only if $\mu_i\geq\mu_{i+1}$ holds for all $i\geq1$.
A sequence of generalized partitions $\mub=\{\mu^{(k)}\}_{k\geq1}$ is called a plane partition
with the boundary conditions
$(\al,\beta,\gamma)$ if and only if the following conditions hold.
\begin{align}
&\mu^{(k)}_i\geq\mu^{(k+1)}_i,\label{MU1}\\
&\lim_{i\rightarrow\infty}\mu^{(k)}_i=\al_k,\notag \\ 
&\mu^{(k)}_i=\infty\ \text{ if $1\leq i\leq \beta_k$}, \notag \\ 
&\lim_{k\rightarrow\infty}\mu^{(k)}_i=\gamma_i.\notag 
\end{align}
We denote by $\mathcal P[\al,\beta,\gamma]$ the set of $\mub$ satisfying these conditions.
For each $\mub$ we define a subset $Y_\mub\subset\bigl(\Z_{\geq1}\bigr)^3$ by
\begin{align}
(i,j,k)\in Y_\mub\leftrightarrow j\leq\mu^{(k)}_i.\label{Y3}
\end{align}
This definition is a generalization of $Y_\lab$ when $\al=\beta=\gamma=\emptyset$.
A new feature 
is that $Y_\mub$ can be an infinite set.
If $\al$ is non-zero for $\mub$, then $Y_\mub$ has an elevation in the $i$-axis.
Similarly, if $\beta$ (resp., $\gamma$) is non-zero, an elevation in the $j$-axis (resp., $k$-axis)
(see Figure \ref{PPELEVATION}).

Plane partitions $\mub$ with the boundary conditions $(\al,\beta,\ga)$
are in one-to-one correspondence with
sets of partitions
$\lab=(\la^{(1)},\la^{(2)},\la^{(3)},\ldots)$.

Set
\begin{align}
&a_k=\al_k-\al_{k+1},\ b_k=\beta_k-\beta_{k+1},\ c_k=\gamma_k,\label{ABC}\\
&\la^{(k)}_i=\mu^{(k)}_{i+\beta_k}-\al_k.\label{LAMU}
\end{align}
The condition \eqref{MU1} for $\mub$ and the condition \eqref{ELEVATION}
for $\lab$ are equivalent through \eqref{ABC} and \eqref{LAMU}.

We fix the parameter $u_i$ as
\begin{align*}
u_i=uq_1^{\al_i}q_2^{i-1}q_3^{\beta_i},
\end{align*}
which implies \eqref{UU}. 
When we discuss the tensor product we use 
\begin{align*}
|\lab\rangle=|\la^{(1)}\rangle\otimes|\la^{(2)}\rangle\otimes|\la^{(3)}\rangle\otimes\cdots
\subset\mathcal F(u_1)\otimes\mathcal F(u_2)\otimes\mathcal 
F(u_3)\otimes \cdots\,,
\end{align*}
and when we discuss the plane partition
we use $Y_\mub$. We show this correspondence by denoting $\lab=\lab_\mub$ when it is necessary.

Consider a linear
subspace of the semi-infinite tensor product 
\begin{align}
&\mathcal M_{\al,\beta,\gamma}\subset\mathcal F(u_1)\otimes\mathcal 
F(u_2)\otimes \cdots\,.
\label{INFTEN}
\end{align}
By definition the space $\mathcal M_{\al,\beta,\gamma}$ is spanned by $|\lab\rangle$
where $\lab=(\la^{(1)},\la^{(2)},\la^{(3)},\ldots)$ is a sequence of partitions satisfying
\eqref{ELEVATION} and the boundary condition.
\begin{align}
&\lim_{k\rightarrow\infty}\la^{(k)}_i=\gamma_i.
\end{align}
The construction of representation with basis $|\lab\rangle$ can be done
by using the result on the finite tensor product, \cite{FFJMM2}, Theorem 3.4.

We consider $|\la^{(i)}\rangle$ as an element of $\mathcal F(u_i)$,
and identify $|\lab\rangle$ with
\begin{align*}
|\la^{(1)}\rangle\otimes|\la^{(2)}\rangle\otimes\cdots\otimes|\la^{(N)}\rangle\in
\mathcal F(u_1)\otimes\mathcal F(u_2)\otimes \cdots\otimes\mathcal F(u_N)
\end{align*}
for large enough $N$. Then, the action of $e(z)$ on $\mathcal M_{\al,\beta,\gamma}$
is the same as in the finite tensor product. 
In the below let us describe the action of $\psi^\pm(z)$ and $f(z)$.

\begin{figure}
\begin{picture}(30,290)(200,-20)

\put(200,235){\line(1,-1){30}}
\put(200,235){\line(-1,-1){40}}
\put(160,195){\line(1,-1){10}}
\put(230,205){\line(-1,-1){10}}
\put(170,185){\line(1,1){20}}
\put(220,195){\line(-1,1){10}}
\put(190,205){\line(1,-1){10}}
\put(210,205){\line(-1,-1){10}}

\put(100,30){\line(0,1){30}}
\put(100,30){\line(1,-1){30}}
\put(100,60){\line(1,-1){10}}
\put(130,0){\line(0,1){15}}
\put(130,15){\line(-1,1){20}}
\put(110,50){\line(0,-1){15}}

\put(300,30){\line(0,1){60}}
\put(300,30){\line(-1,-1){40}}
\put(300,90){\line(-1,-1){10}}
\put(290,80){\line(0,-1){30}}
\put(290,50){\line(-1,-1){20}}
\put(270,30){\line(0,-1){15}}
\put(270,15){\line(-1,-1){10}}
\put(260,5){\line(0,-1){15}}

\put(160,195){\line(0,-1){75}}
\put(100,60){\line(1,1){60}}

\put(170,185){\line(0,-1){75}}
\put(110,50){\line(1,1){60}}

\put(190,205){\line(0,-1){75}}
\put(110,35){\line(1,1){70}}

\put(260,-10){\line(-1,1){70}}
\put(130,0){\line(1,1){60}}

\put(260,5){\line(-1,1){70}}
\put(130,15){\line(1,1){60}}

\put(300,90){\line(-1,1){70}}
\put(230,205){\line(0,-1){45}}

\put(290,80){\line(-1,1){70}}
\put(220,195){\line(0,-1){45}}

\put(290,50){\line(-1,1){80}}
\put(210,205){\line(0,-1){75}}

\put(270,30){\line(-1,1){90}}
\put(200,195){\line(0,-1){75}}

\put(270,15){\line(-1,1){90}}

\put(180,105){\line(0,1){15}}
\put(180,120){\line(1,1){10}}

\put(200,120){\line(1,1){10}}
\put(200,120){\line(-1,1){10}}

\put(190,60){\line(0,1){15}}
\put(220,150){\line(1,1){10}}
\put(170,110){\line(-1,1){10}}

\put(195,212){$\gamma$}
\put(110,22){$\al$}
\put(283,27){$\beta$}

\put(100,200){$Y_\omegab$}
\put(47,5){$i,q_3$}
\put(333,5){$j,q_1$}
\put(205,260){$k,q_2$}
\thicklines
\put(200,235){\vector(0,1){30}}
\put(300,30){\vector(1,-1){30}}
\put(100,30){\vector(-1,-1){30}}
\end{picture}
\caption{Plane partition with the boundary condition
$(\alpha,\beta,\gamma)$. The 
diagram corresponding to the minimal plane partition $\omega$ \eqref{OMEGA} is shown.
}
\end{figure}
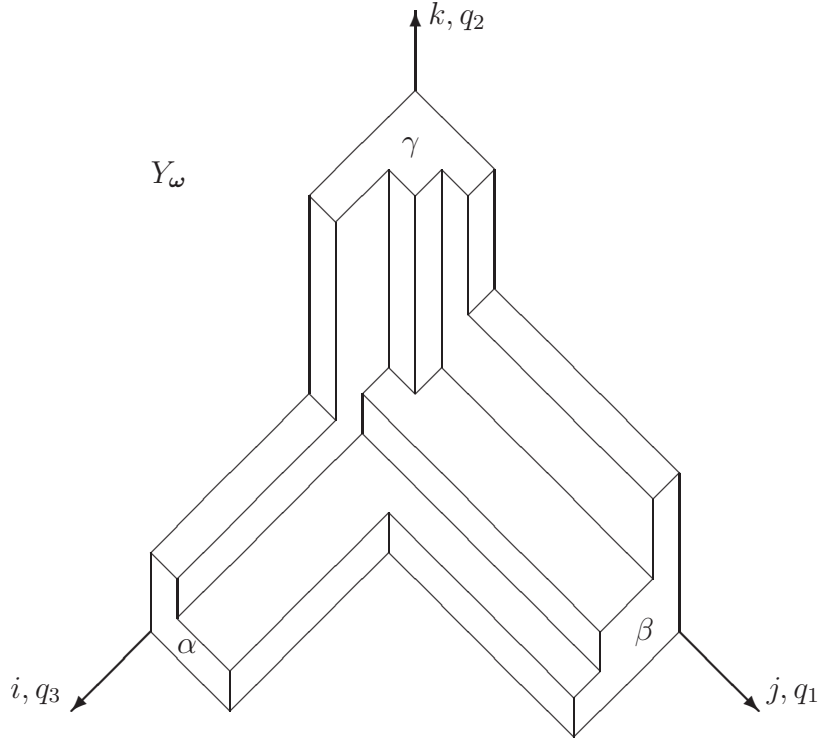
\label{PPELEVATION}

\medskip

We describe the action of $e(z)$.
The action of $e(z)$ adds a box at each concave corner as before
(see \eqref{EACTION}, \eqref{PSIIJK}):
\begin{align*}
&e(z)|\lab\rangle=\sum_{(i,j,k)\in
{\scriptstyle CC(Y_\mub)}}
\psi_{\lab,i,j,k}
\psi_{\la^{(k)},i-\beta_k}\frac1{1-q_1}\delta(q_1^jq_2^kq_3^iu/z)|\lab+1^{(k)}_{i-\beta_k}\rangle.
\end{align*}

Let us discuss the well-definedness of this action.
This point was discussed in the previous subsection in the case of $\al=\beta=\gamma=\emptyset$.
The argument is the same in the
general case, but $\mub$ must be used instead of $\lab$ because
the structure of plane partitions is respected by $\mub$, not by $\lab$ (see \eqref{Y3}).

Recall \eqref{PSI}, and change \eqref{PSIK} to
\begin{align*}
\psi^{(k)}_\lab(u/z)=\prod_{m=1}^k\psi_{\la^{(m)}}(u_m/z),\quad
u_m=q_1^{\al_m}q_2^{m-1}q_3^{\beta_m}u.
\end{align*}
Then using \eqref{LAMU} we obtain
\begin{align}
\psi_{\la^{(m)}}(u_m/z)=\psi_{\mu^{(m)}}(q_2^{m-1}u/z),\label{LAMU2}
\end{align}
where we understand $q_1^\infty=0$. 
Thus, the action of $e(z)$ 
takes the same form as 
in the vacuum case wherein $\lab$ is replaced by $\mub$.  

Define the action of $\psi^\pm(z)$ on
$\mathcal M_{\al,\beta,\gamma}=\mathcal M_{\al,\beta,\gamma}(u,K)$ 
by setting
\begin{align}
&\psi_\lab(u/z)=\frac{1-Ku/z}{1-q_1^{\la^{(1)}_1}u_1/z}\prod_{i=1}^\infty
\frac{1-q_1^{\la^{(i)}_1}q_2u_i/z}{1-q_1^{\la^{(i+1)}_1}u_{i+1}/z}\label{PSIABG}\\
&\times\prod_{j=1}^\infty\frac{1-q_1^{\la^{(1)}_j}q_3^ju_1/z}{1-q_1^{\la^{(1)}_{j+1}}q_3^ju_1/z}
\prod_{i=1}^\infty\prod_{j=1}^\infty
\frac{(1-q_1^{\la^{(i+1)}_j}q_3^ju_{i+1}/z)(1-q_1^{\la^{(i)}_{j+1}}q_2q_3^ju_i/z)}
{(1-q_1^{\la^{(i+1)}_{j+1}}q_3^ju_{i+1}/z)(1-q_1^{\la^{(i)}_j}q_2q_3^ju_i/z)}\nn
\end{align}
in the formula \eqref{Psi1}. This is a finite product. This expression follows from the formal infinite product
\begin{align*}
\psi_\lab(u)=\prod_{i=1}^\infty\psi_{\la^{(i)}}(u_i/z)
\end{align*}
by substituting \eqref{PSI}, and modifying it as we did in \eqref{PSI} and \eqref{LEVELK}.

The function $\psi_\lab(u/z)$, in general, can be better understood
in terms of $Y_\mub$. Let $(i,j,k)\in CC(Y_\mub)$ be a concave corner. Then, adding one box at $(i,j,k)$ to
$Y_\mub$ corresponds to changing $\lab$ to $\lab+1^{(k)}_{i-\beta_k}$,
where $\la^{(k)}_{i-\beta_k}+\al_k+1=j$. From \eqref{LAMU}, 
this relation can be rewritten as
\begin{align*}
\mu^{(k)}_i+1=j.
\end{align*}
Using \eqref{Psi4}, \eqref{PSIABG} we obtain
\begin{align*}
\psi_{\lab+1^{(k)}_{i-\beta_k}}(u/z)=\psi_{i,j,k}(u/z)\psi_\lab(u/z).
\end{align*}
Using \eqref{LAMU2}, for $\lab=\lab_\mub$, we have 
\begin{align}
\psi_\lab(u/z)=\frac{1-Ku/z}{1-u/z}\prod_{(i,j,k)\in Y_\mub}\psi_{i,j,k}(u/z).\label{PSIMAC}
\end{align}
Here the infinite product is defined as follows. Set
\begin{align*}
Y^{(N)}_\mub=\{(i,j,k)\in Y_\mub|i,j,k\leq N\},
\end{align*}
define $\psi^{(N)}_\lab(u/z)$ by \eqref{PSIMAC} with $Y_\mub$ replaced by $Y^{(N)}_\mub$.
For large $N$ the difference between $\psi^{(N)}_\lab(u/z)$ and $\psi^{(N+1)}_\lab(u/z)$
consists only of $N$ dependent factors which 
come
from the $\psi_{i,j,k}(u/z)$ such that
$i\sim N$ or $j\sim N$ or $k\sim N$. We define the infinite product by removing these factors
from $\psi^{(N)}_\lab(u/z)$. By the definition it is independent of $N$. 
Once we define the infinite product in this way, the equality \eqref{PSIMAC}
is clear from \eqref{LAMU2} for $\gamma=\emptyset$. We discuss the case $\gamma\not=\emptyset$
at the end of this subsection.

In fact, it is possible to rewrite the infinite product as a finite product.
We will do it later in Section \ref{SHELL}.
Here we remark that each cube in $Y_\mub$ contributes 
to
poles and zeros through eight
corners of the cube:
poles from $(i,j,k),(i-1,j-1,k),(i-1,j,k-1),(i,j-1,k-1)$ and zeros from $(i-1,j-1,k-1),(i-1,j,k),(i,j-1,k),(i,j,k-1)$;
two of them, $(i,j,k)$ and $(i-1,j-1,k-1)$, cancel each other because of the restriction $q_1q_2q_3=1$.
From this follows that the $\psi^\pm(z)$ action enjoys the $\mathfrak{S}_3$ symmetry
\begin{align}
(i,j,k)\leftrightarrow (j,i,k)
&\Leftrightarrow\ (q_1,q_2,q_3)\leftrightarrow (q_3,q_2,q_1),\label{QSYM13}\\
(i,j,k)\leftrightarrow (k,j,i)
&\Leftrightarrow\ (q_1,q_2,q_3)\leftrightarrow (q_1,q_3,q_2).\label{QSYM23}
\end{align}
The first line means the following. If we transform $Y_\mub$ where
$\mub\in\mathcal P[\al,\beta,\gamma]$, by the involution $(i,j,k)\leftrightarrow (j,i,k)$,
we obtain $Y_{\tilde\mub}$ where $\tilde\mub\in\mathcal P[\beta,\al,\gamma']$.
Set $\lab=\lab_\mub$, and $\tilde\lab=\lab_{\tilde\mub}$.
Then, we have the equality
\begin{align*}
\psi_\la(u/z)\,|_{(q_1,q_2,q_3)\rightarrow (q_3,q_2,q_1)}=\psi_{\tilde\lab}(u/z).
\end{align*}
Similarly, from the involution $(i,j,k)\leftrightarrow (j,i,k)$, we have
\begin{align*}
\psi_\la(u/z)\,|_{(q_1,q_2,q_3)\rightarrow (q_1,q_3,q_2)}=\psi_{\tilde{\tilde\lab}}(u/z),
\end{align*}
where $\tilde{\tilde\lab}=\lab_{\tilde{\tilde\mub}}$ and
$\tilde{\tilde\mub}\in\mathcal P[\gamma,\beta',\al]$.

Define a plane partition
$\omegab=\{\omega^{(k)}\}_{k\geq1}$ 
with the boundary condition $(\al,\beta,\gamma)$
by
\begin{align}
\omega^{(k)}_i=\begin{cases}
\infty&\text{ if $i\leq\beta_k$;}\\
\max(\gamma_i,\al_k)&\text{ otherwise.}\label{OMEGA}
\end{cases}
\end{align}
Then we have
\begin{align*}
\omega^{(k)}_i=
\begin{cases}
\omega^{(k)}_{i+1}&\text{ if $\omega^{(k)}_i=\al_k$};\\
\omega^{(k+1)}_i&\text{ if $\omega^{(k)}_i=\gamma_i$}.
\end{cases}
\end{align*}
Among all $\mub\in{\mathcal P}[\al,\beta,\gamma]$,  $Y_\omegab\subset Y_\mub$ is the minimum.
There is no convex corner in $Y_\omegab$. See Figure \ref{PPELEVATION}.
The set of partitions
$\la=\lab_\omegab$ associated with $\boldsymbol\omega$ is given by
\begin{align*}
\la^{(k)}_i=\max(\gamma_{i+\beta_k},\al_k)-\al_k. 
\end{align*}

Let us compute a few examples
of the eigenvalues \eqref{PSIABG} for $\lab=\lab_\omegab$.
\begin{align*}
\begin{matrix}
\al&\beta&\gamma&\psi_\lab(u/z)\\[3pt]
\emptyset&\emptyset&\emptyset&\frac{1-Ku/z}{1-u/z}\\[7pt]
\{1\}&\emptyset&\emptyset&
\frac{(1-K u/z) (1-q_1 q_2
   u/z)}{(1-q_1 u/z) (1-q_2 u/z)}\\[7pt]
   \{2\}&\emptyset&\emptyset&
   \frac{(1-K u/z) \left(1-q_1^2 q_2
   u/z\right)}{\left(1-q_1^2 u/z\right)
   (1-q_2 u/z)}\\[7pt]
\{1\}&\{1\}&\emptyset&
\frac{(1-K u/z) (1-q_1 q_2 q_3
   u/z)}{(1-q_2 u/z) (1-q_1 q_3
   u/z)}\\[7pt]
   \{1\}&  \{1\}&\{1\}&  
   \frac{(1-K u/z) (1-q_1 q_2 q_3
   u/z)^2}{(1-q_1 q_2 u/z) (1-q_1
   q_3 u/z) (1-q_2 q_3 u/z)}
      \end{matrix}
\end{align*}

Finally, we give the action of $f(z)$ on 
$\mathcal M_{\al,\beta,\gamma}(u,K)$:
\begin{align}
&f(z)|\lab\rangle=\sum_{(i,j,k)\in
{\scriptstyle CV(Y_\mub)}}
\psi'_{\lab,i,j,k}
\psi'_{\la^{(k)},i-\beta_k}\frac{q_1}{1-q_1}\delta(q_1^jq_2^kq_3^iu/z)|\lab-1^{(k)}_{i-\beta_k}\rangle,\nn\\
&\psi'_{\lab,i,j,k}=\psi'_\lab\hskip-3pt{^{(k+1)}}(q_1^{-j}q_2^{-k}q_3^{-i}),\nn\\
&\psi'_\lab\hskip-3pt{^{(k)}}(u/z)=\lim_{N\rightarrow\infty}
\prod_{m=k}^N\psi_{\la^{(m)}}(q_2^{m-1}u/z)
\cdot\frac{1-Ku/z}{1-q_1^{\gamma_1}q_2^Nu/z}
\prod_{j=1}^\infty\frac{1-q_1^{\gamma_j}q_2^Nq_3^ju/z}{1-q_1^{\gamma_{j+1}}q_2^Nq_3^ju/z}.\label{KFACTOR}
\end{align}
We do not repeat the argument which assures the well-definedness of this action.
However, we note that the multiplication of the last infinite product is in fact 
a
finite product
corresponding to the convex corners of $\gamma$ and it removes the extra poles and zeros
in the finite tensor product which do not occur in the semi-infinite product.
It is also important to notice that
\begin{align}
\psi_\lab(u/z)=\psi'_\lab\hskip-3pt{^{(1)}}(u/z),
\end{align}
where the left hand side is given by \eqref{PSIABG} and the right hand side by \eqref{KFACTOR}.
From this the equality \eqref{PSIMAC} for non-trivial $\gamma$ follows.

We summarize the result in this and the previous subsections as
\begin{thm}
There is an action of the algebra $\mathcal E$ on 
$\mathcal M_{\al,\beta,\gamma}(u,K)$
induced from the infinite tensor product \eqref{INFTEN}. This is an irreducible, quasifinite
and tame representation.
The level is $(1,K)$ with a generic parameter $K$.
The lowest
weight is given by \Ref{PSIABG}, $($see also \eqref{FINALPSI} below$)$ 
with $\mub=\omegab$ given by \eqref{OMEGA}.
The representations thus obtained admit the $\mathfrak S_3$ symmetry \eqref{QSYM13}, \eqref{QSYM23}.
\end{thm}

\subsection{Shell formula for the action of $\psi^\pm(z)$}\label{SHELL}

For the Fock representation the rational function $\psi_\la(u/z)$ is factorized into the contribution from
the concave and convex corners. Let us derive the three dimensional version of this statement
for the Macmahon representations. It follows from \eqref{PSIMAC}.

Define an auxiliary object
\begin{align*}
\Psi_\lab(u/z)=\prod_{(i,j,k)\in Y_\mub}\psi_{i,j,k}(u/z),
\end{align*}
where we consider $q_1,q_2,q_3$ as free, i.e., we do not require $q_1q_2q_3=1$.
Once we obtain $\Psi_\lab(u/z)$ as a 
finite product we get $\psi_\lab(u/z)$
by
\begin{align}
\psi_\lab(u/z)=\frac{1-Ku/z}{1-u/z}\Psi_\lab(u/z),\label{PSISHELL}
\end{align}
where $q_1q_2q_3=1$ is imposed.
Let us define the shell of $Y_\mub$ by
\begin{align*}
&\mathcal S_\mub=\{(i,j,k)\in\Z^3\,|\, i,j,k\geq0,\\
&(i+1,j+1,k+1)\not\in Y_\mub,\\
&\{(i,j,k),(i+1,j,k),(i,j+1,k),(i,j,k+1),\\&\quad
(i+1,j+1,k),(i+1,j,k+1),(i,j+1,k+1)\}
\cap Y_\mub\not=\emptyset\}.
\end{align*}
For example
\begin{align*}
\begin{matrix}
Y_\mub&\mathcal S_\mub\\[5pt]
\{\ \}&\{\ \}\\[5pt]
\{(1,1,1)\}&\{(0,0,1),(0,1,0),(1,0,0),(0,1,1),(1,0,1),(1,1,0),(1,1,1)\}
\end{matrix}
\end{align*}
The rational function $\Psi_\lab(u/z)$ has neither 
a pole nor a zero
at $1-q_1^jq_2^kq_3^iu/z=0$ unless $(i,j,k)\in\mathcal S_\mub$. 
It is also worth noting that
for a fixed $(i,j,k)$ the intersection of $\{i+n,j+n,k+n)\,|\,n\in\Z\}$ with $\mathcal S_\mub$ is at most
one point.

We classify the points in the shell $\mathcal S_\mub$ into $\mathcal S^{(n)}_\mub$ ($-1\leq n\leq2$) where
\begin{align*}
\mathcal S^{(n)}_\mub=\{(i,j,k)\in\mathcal S_\mub\,|\,
\Psi_\lab(u/z) \text{ has a zero of order $n$ at $1-q_1^jq_2^kq_3^iu/z=0$}\}.
\end{align*}
For $(i,j,k)\in\mathcal S_\mub$ and $\varepsilon_1,\varepsilon_2,\varepsilon_3=0,1$, we define
\begin{align*}
A_{i,j,k}(\varepsilon_1,\varepsilon_2,\varepsilon_3)=
\begin{cases}
1&\text{ if $(i+\varepsilon_1,j+\varepsilon_2,k+\varepsilon_3)\in Y_\mub$};\\
0&\text{ otherwise}.
\end{cases}
\end{align*}
According as the set of values given in the form of two matrices
\begin{align*}
&T_{i,j,k}=\begin{pmatrix}A_{i,j,k}(0,1,0)&A_{i,j,k}(1,1,0)\\A_{i,j,k}(0,1,1)&A_{i,j,k}(1,1,1)\end{pmatrix},\\
&B_{i,j,k}=\begin{pmatrix}A_{i,j,k}(0,0,0)&A_{i,j,k}(1,0,0)\\A_{i,j,k}(0,0,1)&A_{i,j,k}(1,0,1)\end{pmatrix},
\end{align*}
the order of zero at $1-q_1^jq_2^kq_3^iu/z=0$ of the rational function $\Psi_\lab(u/z)$ is determined:
\begin{align*}
\begin{matrix}
T_{i,j,k}
&\begin{pmatrix}1&1\\1&0\end{pmatrix}
&\begin{pmatrix}1&0\\0&0\end{pmatrix}
&\begin{pmatrix}1&1\\0&0\end{pmatrix}
&\begin{pmatrix}1&0\\1&0\end{pmatrix}
&\begin{pmatrix}1&0\\0&0\end{pmatrix}\\
\\
B_{i,j,k}
&\begin{pmatrix}1&1\\1&1\end{pmatrix}
&\begin{pmatrix}1&1\\1&1\end{pmatrix}
&\begin{pmatrix}1&1\\1&0\end{pmatrix}
&\begin{pmatrix}1&1\\1&0\end{pmatrix}
&\begin{pmatrix}1&1\\1&0\end{pmatrix}\\
\\
\text{order of zero}&-1&1&1&1&2\\
\end{matrix}
\end{align*}
\begin{align*}
\begin{matrix}
T_{i,j,k}
&\begin{pmatrix}1&0\\0&0\end{pmatrix}
&\begin{pmatrix}0&0\\0&0\end{pmatrix}
&\begin{pmatrix}0&0\\0&0\end{pmatrix}
&\begin{pmatrix}1&0\\0&0\end{pmatrix}\\
\\
B_{i,j,k}
&\begin{pmatrix}1&0\\1&0\end{pmatrix}
&\begin{pmatrix}1&0\\0&0\end{pmatrix}
&\begin{pmatrix}1&1\\1&0\end{pmatrix}
&\begin{pmatrix}1&1\\0&0\end{pmatrix}\\
\\
\text{order of zero}&1&1&1&-1\\
\end{matrix}\\
\end{align*}
The cases not listed in this table are neither poles nor zeros.

For any $\al,\beta,\gamma$
and $\mub\in\mathcal P[\al,\beta,\gamma]$ the union of $\mathcal S^{(a)}_\mub$ $(a=-1,1,2)$
is finite. Thus, the formula \eqref{PSISHELL} becomes a finite product
\begin{align}
\psi_\lab(u/z)=(1-Ku/z)\prod_{a=-1,1,2}\prod_{(i,j,k)\in{\textstyle\mathcal S^{(a)}_\mub}}(1-q_1^jq_2^kq_3^iu/z)^a.
\label{FINALPSI}
\end{align}

\subsection{Resonance and submodules}
Now we utilize the factor $1-Ku/z$ (see \eqref{KFACTOR}) in the action of $f(z)$.
Consider the specialization of the level
\begin{align}
K=q_1^bq_2^cq_3^a=q_2^mq_3^n.\label{KVALUE}
\end{align}
Note that
\begin{align*}
m=c-b,n=a-b.
\end{align*}
At this point, the $\mathcal E$-module $\mathcal M_{\al,\beta,\gamma}(u,K)$
is reducible. We denote it by 
$\mathcal M^{m,n}_{\al,\beta,\gamma}(u)$.
In fact, the module $\mathcal M^{m,n}_{\al,\beta,\gamma}(u)$
contains an infinite sequence of submodules. Let us describe these submodules.

Let $\omegab$ be the minimum configuration in $\mathcal P[\al,\beta,\gamma]$ (see \eqref{OMEGA}).
Recall $q_1q_2q_3=1$. For each triple $(\al,\beta,\gamma)$ and $K$ of the form \eqref{KVALUE},
we determine a unique 
$(a,b,c)$ with $a,b,c\ge1$,
satisfying
\eqref{KVALUE} and
\begin{align*}
(a,b,c)\not\in Y_\omegab,
\quad 
(a-1,b-1,c-1)\in Y_\omegab\,. 
\end{align*}
The action of $f(z)$ on 
$\mathcal M^{m,n}_{\al,\beta,\gamma}(u)$
is such that
removing a box at
\begin{align}
(i,j,k)=(a+t,b+t,c+t)\quad(t\in\Z_{\geq0})\label{IJKABC}
\end{align}
is prohibited. This is because the coefficient of $|\lab-1^{(k)}_i\rangle$ in $f(z)|\lab\rangle$
where $(i,j,k)\in CV(Y_\mub)$ ($\lab=\lab_\mub$) and $\la^{(k)}_i=j$,
contains the factor $(1-Ku/z)\delta(q_1^jq_2^kq_3^iu/z)$, but does not contain poles at
$q_1^jq_2^kq_3^iu/z=1$. The poles may appear only if for some $s\geq0$,
$(i+1+s,j+s,k+s)\in Y_\mub$ or $(i+s,j+1+s,k+s)\in Y_\mub$ or $(i+s,j+s,k+1+s)\in Y_\mub$.
However, this is not possible if $(i,j,k)\in CV(Y_\mub)$.
Therefore if the position of the box
is of the form \eqref{IJKABC} the coefficient vanishes when $K$ is specialized as \eqref{KVALUE}.

The lowest
weight vector $|\lab_\omegab\rangle$
is still cyclic in $\mathcal M^{m,n}_{\al,\beta,\gamma}(u)$.
There is no $K$ in the action of $e(z)$. The above consideration tells us that
once 
a box is added
at $(i,j,k)$ of the form \eqref{IJKABC}, one cannot remove it by the action of $f(z)$.
In fact, we will show that the module 
$\mathcal M^{m,n}_{\al,\beta,\gamma}(u)$
contains an infinite series of singular vectors.

Define $\omegab_t=(\omega_t^{(1)},\omega_t^{(2)},\ldots)\in\mathcal P[\al,\beta,\gamma]$ ($t\in\Z_{\geq0}$) by
\begin{align*}
\omega_{t,i}^{(k)}=
\begin{cases}
\max(b+t-1,\omega^{(k)}_i)&\text{ if $i\leq a+t-1$ and $k\leq c+t-1$};\\
\omega^{(k)}_i&\text{ otherwise}.
\end{cases}
\end{align*}
This is the minimal configuration among $\mub\in\mathcal P[\al,\beta,\gamma]$
such that $(a+t-1,b+t-1,c+t-1)\in Y_\mub$. 
Note that $\omegab_0=\omegab$ and
\begin{align*}
Y_{\omegab_0}\subset Y_{\omegab_1}\subset Y_{\omegab_2}\subset\cdots.
\end{align*}
For $t\geq1$ we have
\begin{align*}
CV(\omegab_t)=\{(a+t-1,b+t-1,c+t-1)\}.
\end{align*}

Set
\begin{align*}
\mathcal M^{{m,n},t}_{\al,\beta,\gamma}(u)
=\bigoplus_{Y_\mub\supset Y_{\omegab_t}}\C|\lab_\mub\rangle.
\end{align*}
This is a submodule of
$\mathcal M^{m,n}_{\al,\beta,\gamma}(u)$
with the lowest
vector $|\lab_{\omegab_t}\rangle$ satisfying
\begin{align*}
f(z)|\lab_{\omegab_t}\rangle=0.
\end{align*}
We have the inclusions
\begin{align*}
\mathcal M^{m,n}_{\al,\beta,\gamma}(u)
=\mathcal M^{{m,n},0}_{\al,\beta,\gamma}(u)
\supset
\mathcal M^{{m,n},1}_{\al,\beta,\gamma}(u)
\supset
\mathcal M^{{m,n},2}_{\al,\beta,\gamma}(u)
\supset\cdots.
\end{align*}

In this subsection we study the quotient
\begin{align*}
\mathcal N^{m,n}_{\al,\beta,\gamma}(u)=
\mathcal M^{m,n}_{\al,\beta,\gamma}(u)
/\mathcal M^{{m,n},1}_{\al,\beta,\gamma}(u).
\end{align*}
As we explained $(a,b,c)$ is uniquely determined
once $(\mathbf{a},\mathbf{b},\mathbf{c})$ 
and $m,n$ are fixed.

\begin{prop}
The module $\mathcal N^{m,n}_{\al,\beta,\gamma}(u)$ is an irreducible, quasifinite, tame 
$\mathcal E$-module of level $K=q_2^mq_3^n$. It has a basis parameterized by the set
\be
P_{\al,\beta,\gamma}^{m,n}=\{\mub \in P_{\al,\beta,\gamma}, (a,b,c)\not \in Y_{\mub}\}.
\en
The lowest weight is given by \Ref{PSIABG}, or \eqref{FINALPSI},  
with $\mub=\omegab$ given by \eqref{OMEGA}.
\end{prop}

\subsection{The case of tensor product of the Fock spaces}
Set
\begin{align*}
\overline Y_\omega=\{(i,j,k)\in
\Z^3\mid
i\leq0\text{ or }j\leq0\text{ or }k\leq0\}\sqcup Y_\omega.
\end{align*}
In this subsection we consider the case where the following condition is satisfied
for some $a,b,c$:
\begin{align}\label{splits}
(a-1,b-1,s),(a-1,s,c-1),(s,b-1,c-1)\in\overline Y_\omega, \qquad (s\in\Z_{>0}).
\end{align}
Then, each of the Young diagrams $Y_\al,Y_\beta,Y_\gamma$ contains the following rectangle,
\begin{align*}
&Y_\al\supset C_\al=\{(k,j)\ |\ 1\leq k\leq c-1,\,1\leq j\leq b-1\},\\
&Y_\beta\supset C_\beta=\{(k,i)\ |\ 1\leq k\leq c-1,\,1\leq i\leq a-1\},\\
&Y_\gamma\supset C_\gamma=\{(i,j)\ |\ 1\leq i\leq a-1,\,1\leq j\leq b-1\},
\end{align*}
and each of $\al,\beta,\gamma$ splits into three parts; core, arms and legs.
We define partitions which determine arms and legs of $\al,\beta,\gamma$ as follows:
\begin{align*}
&\al_{\rm arms}=(\al_1-b+1,\ldots,\al_{c-1}-b+1),\\
&\al_{\rm legs}=(\al'_1-c+1,\ldots,\al'_{b-1}-c+1),\\
&\beta_{\rm arms}=(\beta_1-a+1,\ldots,\beta_{c-1}-a+1),\\
&\beta_{\rm legs}=(\beta'_1-c+1,\ldots,\beta'_{a-1}-c+1),\\
&\gamma_{\rm arms}=(\gamma_1-b+1,\ldots,\gamma_{a-1}-b+1),\\
&\gamma_{\rm legs}=(\gamma'_1-a+1,\ldots,\gamma'_{b-1}-a+1).
\end{align*}

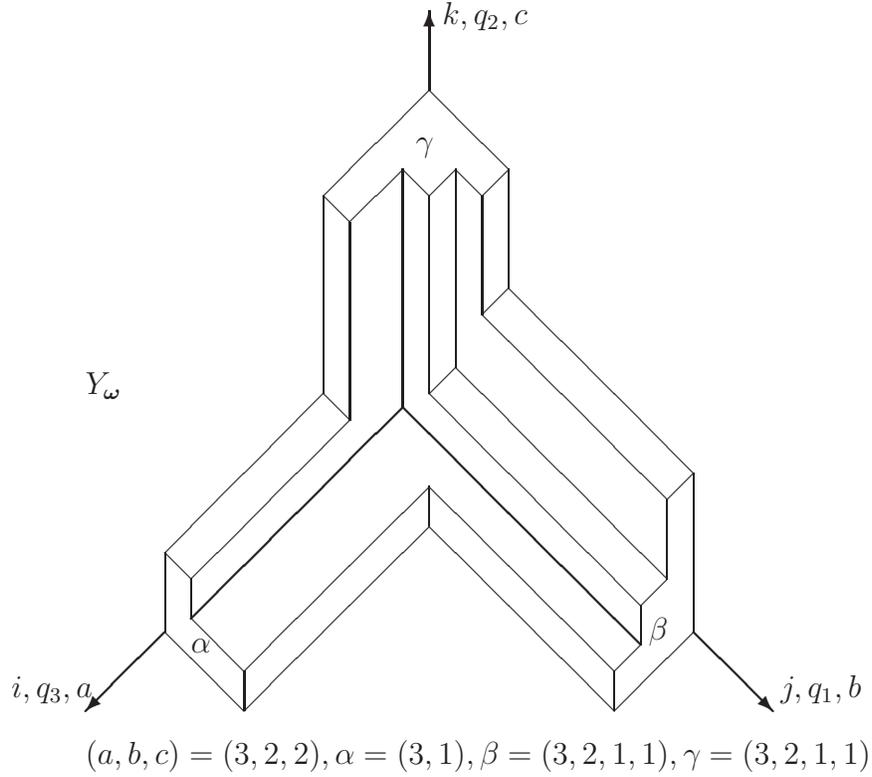
\begin{figure}
\begin{picture}(30,290)(200,-20)

\put(200,235){\line(1,-1){30}}
\put(200,235){\line(-1,-1){40}}
\put(160,195){\line(1,-1){10}}
\put(230,205){\line(-1,-1){10}}
\put(170,185){\line(1,1){20}}
\put(220,195){\line(-1,1){10}}
\put(190,205){\line(1,-1){10}}
\put(210,205){\line(-1,-1){10}}

\put(100,30){\line(0,1){30}}
\put(100,30){\line(1,-1){30}}
\put(100,60){\line(1,-1){10}}
\put(130,0){\line(0,1){15}}
\put(130,15){\line(-1,1){20}}
\put(110,50){\line(0,-1){15}}

\put(300,30){\line(0,1){60}}
\put(300,30){\line(-1,-1){30}}
\put(300,90){\line(-1,-1){10}}
\put(290,80){\line(0,-1){30}}
\put(290,50){\line(-1,-1){10}}
\put(280,40){\line(0,-1){15}}
\put(280,25){\line(-1,-1){10}}
\put(270,15){\line(0,-1){15}}

\put(160,195){\line(0,-1){75}}
\put(100,60){\line(1,1){60}}

\put(170,185){\line(0,-1){75}}
\put(110,50){\line(1,1){60}}

\put(270,0){\line(-1,1){70}}
\put(130,0){\line(1,1){70}}

\put(270,15){\line(-1,1){70}}
\put(130,15){\line(1,1){70}}

\put(300,90){\line(-1,1){70}}
\put(230,205){\line(0,-1){45}}

\put(290,80){\line(-1,1){70}}
\put(220,195){\line(0,-1){45}}

\put(290,50){\line(-1,1){80}}
\put(210,205){\line(0,-1){75}}

\put(280,40){\line(-1,1){80}} 
\put(200,195){\line(0,-1){75}}

\put(200,120){\line(1,1){10}}

\put(200,70){\line(0,1){15}}
\put(220,150){\line(1,1){10}}
\put(170,110){\line(-1,1){10}}

\put(195,212){$\gamma$}
\put(110,22){$\al$}
\put(283,27){$\beta$}
\put(70,-20){$(a,b,c)=(3,2,2),\al=(3,1),\beta=(3,2,1,1),\gamma=(3,2,1,1)$}
\put(70,120){$Y_\omegab$}

\put(42,5){$i,q_3,a$}
\put(333,5){$j,q_1,b$}
\put(205,260){$k,q_2,c$}

\thicklines
\put(190,205){\line(0,-1){90}}
\put(110,35){\line(1,1){80}} 
\put(280,25){\line(-1,1){90}} 

\put(200,235){\vector(0,1){30}}
\put(300,30){\vector(1,-1){30}}
\put(100,30){\vector(-1,-1){30}}
\end{picture}
\caption{The case of tensor product}
\end{figure}

Introduce the notation 
\begin{align*}
\mathcal M^{2,(n)}_{\al,\beta}(u)&=\mathcal M^{(n)}_{\ab,\bb}(u)=\mathcal M^{(n)}_{\ab,\bb}(u;q_1,q_2,q_3),\\
 \mathcal M^{1,(n)}_{\al,\beta}(u)&=\mathcal M^{(n)}_{\ab,\bb}(u;q_3,q_1,q_2),\\
 \mathcal M^{3,(n)}_{\al,\beta}(u)&=\mathcal M^{(n)}_{\ab,\bb}(u;q_2,q_3,q_1).
\end{align*}
Here $\al,\beta$ are related to $\ab,\bb$ by \Ref{ABC}. For example,
\begin{align*}
\mathcal M^{1,(1)}_{r,s}(u)\simeq\mathcal F_2(q_1^rq_3^su),\quad
\mathcal M^{3,(1)}_{r,s}(q_3u)\simeq\mathcal F_3(q_1^sq_2^rq_3u).
\end{align*}
Let us consider a few examples.
The simplest case is $(a,b,c)=(1,1,1)$ and $(\al,\beta,\gamma)=(\emptyset,\emptyset,\emptyset)$.
In this case $\mathcal N^{0,0}_{\emptyset,\emptyset,\emptyset}(u)$
is the trivial 1 dimensional module,
\begin{align*}
\mathcal N^{0,0}_{\emptyset,\emptyset,\emptyset}(u)
\simeq\C,\
K=1,\ \psi_{\lab_\omegab}(u/z)=1.
\end{align*}
We have other specializations for the same $(\al,\beta,\gamma)=(\emptyset,\emptyset,\emptyset)$:
\begin{align*}
&
\mathcal N^{r,0}_{\emptyset,\emptyset,\emptyset}(u)
\simeq\mathcal M^{2,(r)}_{\mathbf{0},\mathbf{0}}(u)\subset
\mathcal F_2(u)\otimes\mathcal F_2(q_2u)\otimes\cdots\otimes\mathcal F_2(q_2^{r-1}u),\\
& K=q_2^r,\ \psi_{\lab_\omegab}(u/z)=\frac{1-q_2^ru/z}{1-u/z},\ (a,b,c)=(1,1,r+1)
\quad(r\geq1).
\end{align*}

If $(\al,\beta,\gamma)=(\emptyset,(1),\emptyset)$ 
then we have the following cases; we consider the cases up to the symmetry.
\begin{align*}
& \mathcal N^{r,0}_{\emptyset,(1),\emptyset}(u)
\simeq
\mathcal M^{2,(r)}_{\mathbf{0},(1,0,\ldots,0)}(u)\subset\mathcal F_2(q_3u)\otimes
\mathcal F_2(q_2u)\otimes\cdots\otimes\mathcal F_2(q_2^{r-1}u),\\
&K=q_2^r,\
\psi_{\lab_\omegab}(u/z)=\frac{1-q_2^ru/z}{1-q_2u/z}\cdot\frac{1-q_2q_3u/z}{1-q_3u_z},\ (a,b,c)=(1,1,r+1)\quad(r\geq1);\\
&\mathcal N^{1,1}_{\emptyset,(1),\emptyset}(u)
\simeq
\mathcal F_2(q_3u)\otimes\mathcal F_3(q_2u),\\
&K=q_2q_3,\ \psi_{\lab_\omegab}(u/z)=\frac{(1-q_2q_3u/z)^2}{(1-q_2u/z)(1-q_3u/z)},\ 
(a,b,c)=(2,1,2).
\end{align*}

If $(\al,\beta,\gamma)=((1),(1),\emptyset)$ then we have the following cases up to the symmetry.
\begin{align*}
&\mathcal N^{r,0}_{(1),(1),\emptyset}(u)
\simeq
\mathcal M^{(r)}_{(1,0,\ldots,0),(1,0,\ldots,0)}(u)\subset
\mathcal F_2(q_1q_3u)\otimes
\mathcal F_2(q_2u)\otimes\cdots\otimes\mathcal F_2(q_2^{r-1}u),\\
&K=q_2^r,\
\psi_{\lab_\omegab}(u/z)=\frac{1-q_2^ru/z}{1-q_2u/z}\cdot\frac{1-q_1q_2q_3u/z}{1-q_1q_3u/z},\
(a,b,c)=(1,1,r+1)
\quad(r\geq1);\\
&\mathcal N^{1,1}_{(1),(1),\emptyset}(u)
\simeq
\mathcal F_2(q_1q_3u)\otimes\mathcal F_3(q_2u),\\
&K=q_2q_3,\ \psi_{\lab_\omegab}(u/z)=\frac{1-q_2q_3u/z}{1-q_2u/z}
\frac{1-q_1q_2q_3u/z}{1-q_1q_3u/z},\ (a,b,c)=(2,1,2).
\end{align*}

If $(\al,\beta,\gamma)=((1),(1),(1))$ then we have the following cases up to the symmetry.
\begin{align*}
&\mathcal N^{0,0}_{(1),(1),(1)}(u)
\simeq
\mathcal F_2(q_1q_3u)\otimes\mathcal F_3(q_1q_2u)\otimes\mathcal F_1(q_2q_3u),\\
& K=1,\
\psi_{\lab_\omegab}(u/z)=\frac{(1-q_1q_2q_3u/z)^3}{(1-q_1q_2u/z)(1-q_1q_3u/z)(1-q_2q_3u/z)},\ 
(a,b,c)=(2,2,2);\\
& \mathcal N^{0,0}_{(1),(1),(1)}(u)
\simeq
\mathcal F_2(q_1q_3u)\otimes\mathcal F_3(q_1q_2u),\\
&K=q_2q_3,\ \psi_{\lab_\omegab}(u/z)=
\frac{(1-q_1q_2q_3u/z)^2}{(1-q_1q_3u/z)(1-q_1q_2u/z)},\ (a,b,c)=(2,1,2).
\end{align*}
Finally we give the general statement:

\begin{prop}
Under condition \Ref{splits}, we have
\begin{align*}
\mathcal N^{c-b,a-b}_{\al,\beta,\gamma}(u)
\simeq
\mathcal M^{2,(c-1)}_{\al_{\rm arms},\beta_{\rm arms}}(q_1^{b-1}q_3^{a-1}u)\otimes
\mathcal M^{3,(a-1)}_{\beta_{\rm legs},\gamma_{\rm arms}}(q_1^{b-1}q_2^{c-1}u)\otimes
\mathcal M^{1,(b-1)}_{\al_{\rm legs},\gamma_{\rm legs}}(q_2^{c-1}q_3^{a-1}u).
\end{align*}
\end{prop}

\section{$\glinf$-modules and Genfand-Zetlin basis}\label{GZbasis}

\subsection{Algebra $\glinf$}

In this section, we introduce a family of $\glinf$-modules 
which arises as a limit of $\E$-modules considered 
in the previous section.

We fix the notation as follows. 
By definition, $\glinf$ is the complex Lie algebra 
with basis $\{E_{i,j}\}_{i,j\in\Z}$ and the commutation relations
$[E_{i,j},E_{k,l}]=\delta_{j,k}E_{i,l}-\delta_{l,i}E_{k,j}$. 
We set
\begin{align*}
E_i=E_{i,i+1},\quad F_i=E_{i+1,i},\quad 
H_i=E_{i,i}-E_{i+1,i+1}\,.
\end{align*}
We shall consider also the following Lie subalgebras of $\glinf$,
\begin{align*}
&
\glhf^+=\mathrm{span}\,\{E_{i,j}\mid i,j\ge 1\}\,
\quad
\glhf^-=\mathrm{span}\,\{E_{i,j}\mid i,j\le 0\}\,,
\\
&\mathfrak{g}_{r,s}=\mathrm{span}\,\{E_{i,j}  \mid r\le i,j\le s\} 
\ \simeq\ \gl_{s-r+1}\,,
\end{align*}
where $r,s\in\Z$, $r<s$. 

For a sequence of complex numbers  
$\theta=(\theta_i)_{i\in\Z}$, we denote by 
 $\mathcal{W}_{\theta}$
the unique irreducible $\glinf$-module generated by a vector $v$
such that
\begin{align*}
E_{i,j}v=0\quad (i>j),\qquad
E_{i,i}v=\theta_i v\quad(i\in\Z).
\end{align*} 
The $\theta$ is called the lowest weight and the vector $v$ is called the lowest weight vector.

\subsection{Gelfand-Zetlin basis}
Let $N$ be a positive integer. 
We recall the Gelfand-Zetlin (GZ) basis for irreducible representations
of 
$\mathfrak{g}_{-N+1,0}\simeq\gl_N$. 

A Gelfand-Zetlin (GZ) pattern for $\gl_N$ 
is an array of 
integers
\begin{align}
\mu=\,   %
\begin{matrix}
\mu^{(1)}_{1} &           &          &          \\
\mu^{(2)}_{1} &\mu^{(2)}_{2}  &          &          \\
\vdots    &\ddots     &  \ddots  &          \\
\mu^{(N)}_{1} &\mu^{(N)}_{2}  & \cdots   &\mu^{(N)}_{N} \\
\end{matrix}\,,
\label{GZhk}
\end{align}
such that
\begin{align}
\mu^{(i)}_{j}\ge \mu^{(i)}_{j+1}\,,\quad \mu^{(i)}_{j}\ge
\mu^{(i+1)}_{j}\quad \text{for all $i,j$.}
\label{GZcond}
\end{align}
Quite generally, 
we shall denote by $\mu \pm \one^{(i)}_{j}$ the GZ pattern
obtained by changing $\mu^{(i)}_{j}$ to $\mu^{(i)}_{j}\pm 1$ while 
keeping the rest of the entries unchanged. 

Given a set of integers $\eta=(\eta_1,\cdots,\eta_N)$, 
$\eta_1\ge\cdots\ge\eta_N$, let 
$L_\eta$ be the vector space
with basis $\{|\mu\rangle_{(N)}\}$, 
where $\mu$ runs over all GZ patterns for $\gl_N$ 
satisfying 
\begin{align*}
\mu^{(i)}_{i}=\eta_i \quad (i=1,\cdots,N).
\end{align*}
We set $|\mu\rangle_{(N)}=0$ if the condition \eqref{GZcond} is violated. 

Notation being as above, the 
following formulas define an action of 
$\mathfrak{g}_{-N+1,0}$
on $L_\eta$:
\begin{align}
&E_{-i,-i+1}
\ket{\mu}_{(N)}
=\sum_{j=1}^{N-i} \ket{\mu+\one^{(i+j)}_{j}}_{(N)}\,
\frac{\prod_{k=1}^{N-i+1}(\ell^{(i+j)}_{j}-\ell^{(i-1+k)}_{k})}
{\prod_{1\le k(\neq j)\le N-i}(\ell^{(i+j)}_{j}-\ell^{(i+k)}_{k})}\quad 
(1\leq i\leq N-1),
\label{GZact1}\\
&
E_{-i+1,-i}
\ket{\mu}_{(N)}
=-\sum_{j=1}^{N-i} \ket{\mu-\one^{(i+j)}_{j}}_{(N)}\,
\frac{\prod_{k=1}^{N-i-1}(\ell^{(i+j)}_{j}-\ell^{(i+1+k)}_{k})}
{\prod_{1\le k(\neq j)\le N-i}(\ell^{(i+j)}_{j}-\ell^{(i+k)}_{k})}\quad 
(1\leq i\leq N-1),
\label{GZact2}\\
&E_{-i,-i}
\ket{\mu}_{(N)}
=\bigl(\sum_{j=1}^{N-i}\mu^{(i+j)}_{j}-\sum_{j=1}^{N-i-1}\mu^{(i+1+j)}_{j}\,
\bigr)\ket{\mu}_{(N)}\quad (0\leq i\leq N-1),
\label{xxx}
\end{align}
where
\begin{align}
&\ell^{(i+j)}_{j}=\mu^{(i+j)}_{j}-j+1\,.
\label{ell-ij}
\end{align}
The representation $L_\eta$ is irreducible. 
The highest weight is
$(\theta_{-N+1},\theta_{-N+2},\cdots,\theta_{0})
=(\eta_1,\eta_{2},\cdots,\eta_N)$
and the  lowest weight is $(\eta_N,\eta_{N-1},\cdots,\eta_1)$, 
the corresponding highest (resp. lowest) weight
vector being 
given by the GZ pattern with $\mu^{(i)}_{j}=\eta_j$ 
(resp. $\mu^{(i)}_{j}=\eta_i$) for all $i,j$.  

Now we extend this construction to the case of $\glhf^-$. 
In the following we fix a positive integer $n$. 
Consider an infinite GZ pattern of width $n$, 
\begin{align}
\mu=\ 
\begin{matrix}
\mu^{(1)}_{1} &       &          &                  &          &\\
\vdots    &\ddots &          &                  &          &\\
\mu^{(n)}_{1}&\cdots &\mu^{(n)}_{n}  &                  &          &\\
\mu^{(n+1)}_{1}&\cdots &\mu^{(n+1)}_{n} & 0    &           &  \\
\mu^{(n+2)}_{1} &\cdots &\mu^{(n+2)}_{n} & 0    &0           &    \\
\vdots &\cdots &\vdots  & 0    &0           &0 \cdots\\
\end{matrix}\,
\label{GZhk1}
\end{align}
that is, an array of integers $\mu=(\mu^{(i)}_{j})_{i\ge j\ge1}$ 
satisfying \eqref{GZcond} and
\begin{align}
\mu^{(i)}_{j}=0\quad \text{if $j>n$}\,.
\label{depth}
\end{align}

Let $\eta=(\eta_1,\cdots,\eta_n)$, $\ga=(\ga_1,\cdots,\ga_n)$  
be partitions such that 
$\eta_i\ge \ga_i$, $i=1,\cdots,n$. 
Let $\Yb^-_{\eta,\ga}$ be the vector space with basis 
$\{\ket{\mu}\}$, where
$\mu=(\mu^{(i)}_{j})_{i\ge j\ge 1}$ runs over GZ patterns \eqref{GZhk1}
of width $n$ satisfying the conditions
\begin{align}
&\mu^{(i)}_{i}=\eta_i\quad (i=1,\cdots,n)\,,
\label{bcond1}\\
&\mu^{(i)}_{j}=\ga_j\quad (i\gg 1,\ j=1,\cdots,n)\,.
\label{bcond2}
\end{align}

\begin{prop}
The following formulas define a representation of $\glhf^-$ on 
$\Yb^-_{\eta,\ga}$:
\begin{align}
&E_{-i,-i+1}
\ket{\mu}
=\sum_{j=1}^{n} \ket{\mu+\one^{(i+j)}_{j}}c^+_{i+j,j}(\mu)\,,
\quad E_{-i+1,-i}
\ket{\mu}=\sum_{j=1}^{n} \ket{\mu-\one^{(i+j)}_{j}}
c^-_{i+j,j}(\mu)\quad(i\geq1),
\label{GZact1-}\\
&E_{-i,-i}\ket{\mu}
=\sum_{j=1}^{n} \bigl(\mu^{(i+j)}_{j}-\mu^{(i+1+j)}_{j}\bigr)\,
\ket{\mu}\quad(i\geq0),
\label{GZact2-}
\end{align}
where
\begin{align*}
&c^\pm_{i+j,j}(\mu)
=\pm \frac{\prod_{k=1}^{n}(\ell^{(i+j)}_{j}-\ell^{(i\mp1+k)}_{k})}
{\prod_{1\le k(\neq j)\le n}(\ell^{(i+j)}_{j}-\ell^{(i+k)}_{k})}\,,
\end{align*}
and  $\ell^{(i+j)}_{j}$ is defined by \eqref{ell-ij}.
\end{prop}
\begin{proof}
Clearly the operators 
\eqref{GZact1-}, \eqref{GZact2-}
preserve the space $\Yb^-_{\eta,\ga}$. 
Given $i$, take $N$ so that $N>n+i+1$. 
Then, under the condition \eqref{depth}, 
the formulas \eqref{GZact1}--\eqref{xxx}
reduce to \eqref{GZact1-}, \eqref{GZact2-} 
after making a base change of the form $\ket{\mu}=f(\mu)\ket{\mu}_{(N)}$.
Consider \eqref{GZact1}. The range of summation $1\leq j\leq N-i$
reduces to that in \eqref{GZact1-} $1\leq j\leq n$ because $\mu^{(i+j)}_j=0$
is unchanged. We see also that the coefficient
$f(\mu)$ is to satisfy $f(\mu+\one^{(i+j)}_{j})=f(\mu)(\ell^{(i+j)}_{j}+N-i)$, 
which can be solved easily. 
Hence the commutation relations of the generators are obviously satisfied.
\end{proof}

\begin{prop}\label{irr-gminus}
If $\ga_1=\cdots=\ga_n$, then  $\Yb^-_{\eta,\ga}$ is an irreducible
$\glhf^-$-module. 
\end{prop}
\begin{proof}
For $N>n$, consider the subspace of $\Yb^-_{\eta,\ga}$
\begin{align*}
W_N=\mathrm{span}\{\ket{\mu}\in\Yb^-_{\eta,\ga}\mid 
\mu^{(i)}_{j}=\ga_j \quad (i>N,\ j=1,\cdots,n)\}\,.
\end{align*}
Because of the restriction $\ga_1=\cdots=\ga_n$, the subspace
$W_N$ is invariant under the action of $\mathfrak{g}_{-N+1,0}$.
The vector $\ket{\mu}\in W_N$ defined by $\mu^{(i)}_{j}=\eta_i$ ($1\le i\le n$) and  
$\mu^{(i)}_{j}=\ga_1$ ($i>n$) is a
$\mathfrak{g}_{-N+1,0}$-singular vector with the lowest weight 
$(\theta_{-N+1},\cdots,\theta_{0})=(0,\cdots,0,\eta_n-\ga_1,\cdots,\eta_1-\ga_1)$.
Moreover $W_N$ has the same dimension as that of the irreducible lowest 
weight module of the same lowest weight. To see this one can rearrange the table
as the usual GT pattern:
\smallskip
\begin{align*}
\begin{matrix}
\eta_1&&\cdots&\eta_n&&\gamma_1&&\cdots&&\gamma_1\\
&\mu^{(2)}_1&&&\mu^{(n+1)}_n&&\ddots\\
&&\ddots&&&\ddots&&\gamma_1\\
&&&&&&\mu^{(N)}_n\\
&&&&\ddots\\
&&&&&\mu^{(N)}_1
\end{matrix}
\end{align*}

\smallskip

Hence $W_N$ is $\mathfrak{g}_{-N+1,0}$-irreducible. 

By the definition we have $\Yb^-_{\eta,\ga}=\cup_{N>n}W_N$. 
The irreducibility follows from this. 
\end{proof}

By applying the involutive automorphism 
$\sigma(E_{i,j})=-E_{1-j,1-i}$ of $\glinf$, 
we obtain representations of the subalgebra $\glhf^+=\sigma(\gl^-_{\infty/2})$.
For later reference let us
write the relevant formulas for $\glhf^+$. 
 
For  $\mathfrak{g}_{\infty/2}^+$, 
we use the transposed GZ patterns $\mu=(\mu^{(i)}_{j})_{1\le i\le j}$ 
of depth $n$,  
\begin{align}
\mu=\,
\begin{matrix}
\mu^{(1)}_{1} & \cdots    &\mu^{(1)}_{n} &\mu^{(1)}_{n+1}&
\mu^{(1)}_{n+2}& \cdots\\
     & \ddots    &          &\vdots     &\vdots     &       \\
          &           &\mu^{(n)}_{n} &\mu^{(n)}_{n+1}
&\mu^{(n)}_{n+2}&\cdots \\
          &           &          & 0         & 0         &\cdots \\
          &           &          &           & 0         & \cdots \\
\end{matrix}\,
\label{GZhk2}
\end{align}
Let $\eta=(\eta_1,\cdots,\eta_n)$,
$\al=(\al_1,\cdots,\al_n)$
 be partitions 
such that 
$\eta_i\ge \al_i$, $i=1,\cdots,n$.
We set $\al_i=0$ for $i>n$. Let 
$\Yb^+_{\eta,\al}$ be the vector space with basis $\{\ket{\mu}\}$ where
the $\mu=(\mu^{(i)}_{j})_{1\le i\le j}$ satisfy 
$\mu^{(i)}_{j}=0$ if $i>n$ and 
\begin{align*}
&\mu^{(i)}_{i}=\eta_i\quad (i=1,\cdots,n)\,,
\\
&\mu^{(i)}_{j}=\al_i\quad (i=1,\cdots,n,\ j\gg 1)\,.
\end{align*}

\begin{prop}
The following formulas define a representation of $\glhf^+$ on 
$\Yb^+_{\eta,a}$:
\begin{align}
&E_{i,i+1} 
\ket{\mu}=\sum_{j=1}^n
\ket{\mu+\one^{(j)}_{i+j}}c^+_{j,i+j}(\mu)\,,
\quad
E_{i+1,i}
\ket{\mu}=\sum_{j=1}^n \ket{\mu-\one^{(j)}_{i+j}}c^-_{j,i+j}(\mu)\,,
\label{GZact3}\\
&E_{i,i}
\ket{\mu}=
\bigl(\sum_{j=1}^n \mu^{(j)}_{i+j}-\sum_{j=1}^n\mu^{(j)}_{i-1+j}-n\bigr)
\ket{\mu}\,,
\label{GZact4}
\end{align}
where $i\ge 1$, and 
\begin{align*}
&c^\pm_{j,i+j}(\mu)=\mp 
\frac{\prod_{k=1}^n(\ell^{(j)}_{i+j}-\ell^{(k)}_{i\mp1+k})}
{\prod_{1\le k(\neq j)\le n}(\ell^{(j)}_{i+j}-\ell^{(k)}_{i+k})}\,,\quad
\ell^{(j)}_{i+j}=\mu^{(j)}_{i+j}-j+1\,.
\end{align*}
\end{prop}

We note that it is always possible to twist a given representation 
by changing $E_{i,j}$ to $E_{i,j}+x\delta_{i,j}\cdot\mathrm{id}$ 
for some $x\in \C$. 
Utilizing this freedom we have chosen $x=-n$
in \eqref{GZact4}, which will be convenient in the next subsection.


\subsection{Representations of $\glinf$}

In this subsection, we glue together the representations of $\glhf^\pm$
to define representations of the full algebra $\glinf$. 
Consider now a GZ pattern $\mu=(\mu^{(i)}_{j})_{i,j\ge 1}$ 
such that $\mu^{(i)}_{j}=0$ if $i>n$ and $j>n$,
that is
\begin{align}
\mu=\quad
\begin{matrix}
\mu^{(1)}_{1} &\cdots &\mu^{(1)}_{n} &\mu^{(1)}_{n+1} &\mu^{(1)}_{n+2} &\cdots\\
     \vdots         &\ddots &     \vdots & \vdots           &  \vdots          &\\
\mu^{(n)}_{1}&\cdots &\mu^{(n)}_{n} &\mu^{(n)}_{n+1} &\mu^{(n)}_{n+2} &\cdots\\
\mu^{(n+1)}_{1}&\cdots &\mu^{(n+1)}_{n} & 0    &0           &\cdots\\
\mu^{(n+2)}_{1} &\cdots &\mu^{(n+2)}_{n} & 0    &0           &\cdots\\
\vdots &\cdots &\vdots  & 0    &0           &\cdots\\
\end{matrix}\,
\label{GZhk3}
\end{align}
We say $\mu$ has a hook shape of width $n$.  
We assume further that 
\begin{align}
&\mu^{(i)}_{j}=\al_i\quad \text{for $i=1,\cdots,n$, $j\gg 1$},
\label{stab1}\\
&\mu^{(i)}_{j}=\ga_j\quad \text{for $j=1,\cdots,n$, $i\gg 1$}.
\label{stab2}
\end{align}
Let $\Yb_{\al,\ga}$ be the vector space with basis 
$\{\ket{\mu}\}$, where $\mu$
runs over hook-shape GZ patterns of width $n$, 
satisfying \eqref{stab1}, \eqref{stab2}.

\begin{prop}
Define
\begin{align}
&E_{0,1}
\ket{\mu}=\sum_{j=1}^n \ket{\mu+\one^{(j)}_{j}}c^+_{j,j}(\mu)\,,
\quad
E_{1,0}
\ket{\mu}=\sum_{j=1}^n \ket{\mu-\one^{(j)}_{j}}c^-_{j,j}(\mu)\,,
\label{GZact0}\\
&
c^+_{j,j}(\mu)=\frac{1}{\prod_{k(\neq j)}(\ell^{(j)}_{j}-\ell^{(k)}_{k})}\,,
\quad 
c^-_{j,j}(\mu)
=-
\frac{\prod_{k=1}^n(\ell^{(j)}_{j}-\ell^{(k+1)}_{k})
(\ell^{(j)}_{j}-\ell^{(k)}_{k+1})}
{\prod_{1\le k(\neq j)\le n}(\ell^{(j)}_{j}-\ell^{(k)}_{k})}\,,
\label{GZact00}
\end{align}
with
\begin{align*}
\ell^{(k)}_{k}=\mu^{(k)}_{k}-k+1\,.
\end{align*}
Then the above formulas along with 
\eqref{GZact1-}, \eqref{GZact2-},\eqref{GZact3}, \eqref{GZact4}
give a representation of $\glinf$ on $\Yb_{\al,\ga}$. 
\end{prop}
\begin{proof}
The only non-trivial relations to check are 
$[E_i,F_j]=\delta_{i,j}H_i$
for $i=0$ or $j=0$, and the Serre relations involving them. 
 First we check the former. 
For the relations $[E_0,F_i]=[E_i,F_0]=0$ to hold 
for $i\neq 0$, we must have that
\begin{align*}
&\frac{c^+_{j,j}(\mu+\one^{(k)}_{1+k})}{c^+_{j,j}(\mu)}
=\frac{\ell^{(k)}_{1+k}-\ell^{(j)}_{j}+1}{\ell^{(k)}_{1+k}-\ell^{(j)}_{j}}\,,
\\
&\frac{c^+_{j,j}(\mu+\one^{(1+k)}_{k})}{c^+_{j,j}(\mu)}
=\frac{\ell^{(1+k)}_{k}-\ell^{(j)}_{j}+1}{\ell^{(1+k)}_{k}-\ell^{(j)}_{j}}\,,
\end{align*}
and that 
$c^\pm_{j,j}(\mu\pm\one^{(k)}_{i+k})
=c^\pm_{j,j}(\mu\pm\one^{(i+k)}_{k})=c^\pm_{j,j}(\mu)$ 
in all other cases. 
These relations can be verified using \eqref{GZact00}. 

A similar calculation shows that, in $[E_0,F_0]\ket{\mu}$, all terms cancel except 
\begin{align*}
\sum_{j=1}^n\left(c^-_{j,j}(\la)c^+_{j,j}(\mu+\one^{(j)}_{j})
-c^+_{j,j}(\mu)c^-_{j,j}(\mu-\one^{(j)}_{j})\right)\,\ket{\mu}\,.
\end{align*}
Substituting  \eqref{GZact00} we find that the coefficient in front
of $\ket{\mu}$ can be written as
\begin{align*}
-&\sum_{j=1}^n \left(\mathop\mathrm{res}_{z=\ell^{(j)}_{j}}
+\mathop\mathrm{res}_{z=\ell^{(j)}_{j}-1}\right)\prod_{k=1}^n
\frac{z-\ell^{(k)}_{k+1}+1}{z-\ell^{(k)}_{k}+1}
\frac{z-\ell^{(k+1)}_{k}+1}{z-\ell^{(k)}_{k}}
\\
&=\sum_{k=1}^n(\ell^{(k)}_{k+1}+\ell^{(k+1)}_{k}-2\ell^{(k)}_{k})-n\,.
\end{align*}
Comparing this with $H_0=E_{0,0}-E_{1,1}$
we obtain $[E_0,F_0]\ket{\mu}=H_0\ket{\mu}$.

Finally the Serre relations involving $E_{0}$ or $F_{0}$
can be checked by a 
tedious but straightforward calculation. 
\end{proof}
\begin{prop}\label{GLMODULE}
If $\ga_1=\cdots=\ga_n=c$, 
$\Yb_{\al,\ga}$ is an irreducible $\glinf$-module with the
lowest weight vector $\ket{\mu^{(n)}(\al,c)}$, where
\begin{align*}
\mu^{(n)}(\al,c)^{(i)}_{j}=
\begin{cases}
\max(\al_i,c)& (1\le i,j\le n);\\
\al_i & (1\le i\le n,\ j>n);\\
c & (i>n,\ 1\le j\le n).
\end{cases}
\end{align*}
If $\al_1\ge\cdots\ge \al_k\ge c\ge \al_{k+1}\ge\cdots\ge \al_n$, then 
$\Yb_{\al,\ga}$ is isomorphic to the irreducible lowest weight 
$\glinf$-module
$\mathcal{W}_{\theta^{(n)}(\al,c)}$, with
the lowest weight
\begin{align}
\theta^{(n)}(\al,c)_i=
\begin{cases}
0 & (i\le -k);\\
\al_{-i+1}-c & (-k+1\le i\le 0);\\
\al_{n-i+1}-c-n & (1\le i\le n-k);\\
-n & (i\ge n-k+1).\\
\end{cases}
\label{GZ-hwt}
\end{align}
\end{prop}
\begin{proof}

For $N\ge0$, consider the Lie subalgebra $\mathfrak{a}_N=
\mathfrak{g}_{-\infty,N}$ spanned by $E_{i,j}$ with $i,j\le N$. 
For each partition 
$\eta=(\eta_1,\cdots,\eta_n)$ 
such that $\eta_i\ge c$ if $i+N\le n$,  
we consider the $\mathfrak{a}_N$-module 
$\mathfrak{X}_{N,\eta}$ 
given as follows. 
As a linear space, it is spanned by vectors $\ket{\mu}$ where
$\mu=(\mu^{(i)}_{j})_{j\le i+N}$ runs over GZ patterns of width $n$ such that 
\begin{align*}
&\mu^{(j)}_{N+j}=\eta_j \quad \text{for $j=1,\cdots,n$},\\
&\mu^{(i)}_{j}=c \quad \text{for $1\le j\le n$, $i\gg 1$.}
\end{align*}
The action of the generators of $\mathfrak{a}_N$
is defined by the same 
formula as \eqref{GZact1}--\eqref{GZact4}. 
We show that $\mathfrak{X}_{N,\eta}$ 
is an irreducible 
$\mathfrak{a}_{N}$-module for all $N$ and $\eta$. 
The irreducibility of $\Yb_{\al,\ga}$ 
is a simple consequence of this assertion. 

For $N=0$ we have $\mathfrak{a}_0=\glhf^-$, 
and $\mathfrak{X}_{0,\eta}=\Yb_{\eta,c}^-$ 
is an
irreducible $\mathfrak{a}_{0}$-module 
by Proposition \ref{irr-gminus}. 
Assume by induction that each $\mathfrak{X}_{N-1,\xi}$
is 
$\mathfrak{a}_{N-1}$-irreducible for $N>0$. 
By the definition, 
we have a direct sum decomposition into subspaces
$\mathfrak{X}_{N,\eta}=\oplus_\xi \mathfrak{X}_{N-1,\xi}$, 
where $\xi_1\ge\eta_1\ge\xi_2\ge\eta_2\ge\cdots\ge\xi_n\ge\eta_n$.    
Each $\mathfrak{X}_{N-1,\xi}$ 
is an irreducible 
$\mathfrak{a}_{N-1}$-module, which are mutually inequivalent. 
Therefore, if $W\subset \mathfrak{X}_{N,\eta}$ is  
a non-zero $\mathfrak{a}_{N}$-submodule, then  
we have $W=\oplus_\xi W\cap \mathfrak{X}_{N-1,\xi}$  
as $\mathfrak{a}_{N-1}$-module. 
If $W\cap \mathfrak{X}_{N-1,\xi}\neq 0$, 
then acting with $E_{N-1},F_{N-1}$ 
we obtain that 
$W\cap
 \mathfrak{X}_{N-1,(\xi_1,\cdots,\xi_i\pm1,\cdots,\xi_n)}\neq 0$ for each $i$, 
as long as the condition $\eta_{i-1}\ge\xi_i\pm 1\ge \eta_i$ is 
not violated. 
It is now easy to see that $W=\mathfrak{X}_{N,\eta}$. 
The proof is over. 
\end{proof}
Note that the lowest weight $\theta^{(n)}(\al,c)_i$ is increasing, i.e., ``anti-dominant", except
\begin{align*}
\theta^{(n)}(\al,c)_0=\al_1-c\geq 0>-n\geq\theta^{(n)}(\al,c)_1=\al_n-c-n.
\end{align*}

\subsection{Degeneration of the algebra $\mathcal{E}$ and $\glinf$}
In this subsection we 
examine the degeneration of the algebra $\E$ and its modules 
when one of the parameters $q_i$ tends to $1$.

In order to discuss the limit, it is convenient to introduce the
elements $h_m\in \E_{q_1,q_2,q_3}$ ($m \neq0$) via
\begin{align*}
\psi^\pm(z)=\psi^\pm_0
\exp\Bigl(\mp \sum_{\pm m>0}\frac{\ga_m}{m} h_mz^{-m}\Bigr)\,,
\quad \ga_m=\prod_{i=1}^3(1-q_i^m)\,.
\end{align*}
We have $[h_m,e(z)]=z^m\, e(z)$, $[h_m,f(z)]=-z^m\, f(z)$.
Set further $\psi^+_0=q_1^{\kappa_+}$, $\psi^-_0=q_1^{\kappa_-}$.
In the limit 
\begin{align}
q_1\to1\,,\quad  q_2\to q\,,\quad q_3\to q^{-1}\,
\quad (q\in\C^\times), 
\label{limit}
\end{align}
the algebra $\E_{q_1,q_2,q_3}$ reduces
to the Lie algebra $\mathfrak{d}_{q,\kappa,0}$ 
which has been mentioned already.  
The algebra $\mathfrak{d}_{q,\kappa,0}$
is the associative algebra (viewed as a Lie algebra) generated
by $Z^{\pm1},D^{\pm1}$ with $DZ=qZD$, extended by 
a central element $\kappa$:
\begin{align*}
[Z^{i_1}D^{j_1},Z^{i_2}D^{j_2}]=(q^{j_1i_2}-q^{j_2i_1})Z^{i_1+i_2}D^{j_1+j_2}
+ i_1q^{-i_1j_1}\delta_{i_1+i_2,0}\delta_{j_1+j_2,0}\cdot\kappa\,.
\end{align*}
Writing the limit of the generators $e_m,f_m,h_m$ with bars,  
we have the identification 
\begin{align}
&(1-q)\bar{e}_m=D^mZ,\quad
-(1-q^{-1})\bar{f}_m=Z^{-1}D^m,\,
\label{bar1}\\
&(1-q^{-m})\bar{h}_m=D^m\quad (m\neq 0),\quad
\kappa_+-\kappa_-=\kappa\,.
\label{bar2}
\end{align}

Let $\gl_{\infty,\kappa}$ be the Lie algebra
defined by the symbols $:E_{i,j}:$
and a central element $\kappa$, with relations
\begin{align*}
&\gl_{\infty,\kappa}=\{\sum_{i,j\in\Z}a_{i,j}:E_{i,j}:\mid \exists 
N>0,\ a_{i,j}=0\ \text{ for $|i-j|>N$}\}
\oplus \C\kappa\,,
\\
&[\sum_{i,j}a_{i,j}:E_{i,j}:,\sum_{k,l}b_{k,l}:E_{k,l}:]
=\sum\Bigl(\sum_{k}a_{i,k}b_{k,j}-\sum_k b_{i,k}a_{k,j}\Bigr):E_{i,j}:
\\
&\qquad\qquad\qquad\qquad\qquad\qquad\quad+
\Bigl(\sum_{i\le 0<j} 
a_{i,j}b_{j,i}
-\sum_{i> 0\ge j}
a_{i,j}b_{j,i}\Bigr)\kappa\,.
\end{align*}
It is straightforward to verify that the map
\begin{align}
&D^m Z\mapsto \sum_{i\in\Z}
:E_{i,i+1}:
u^mq^{-im}\,,
\quad 
Z^{-1}D^m\mapsto \sum_{i\in\Z}
:E_{i+1,i}:
u^mq^{-im}\,,
\label{lim-h1}\\
&D^m\mapsto 
\sum_{i\in \Z}:E_{i,i}:u^m q^{-im}
+\frac{u^m}{1-q^m}\kappa\quad(m\neq0),
\quad \kappa\mapsto \kappa\,
\label{lim-h2}
\end{align}
gives an embedding of Lie algebras
\begin{align*}
\iota_u~:~\mathfrak{d}_{q,\kappa}\longrightarrow
 \mathfrak{gl}_{\infty,\kappa}\,.
\end{align*}
Here $u$ is an arbitrary non-zero complex number. 

We view $\gl_\infty$ as a subalgebra of $\gl_{\infty,\kappa}$ 
by $E_{i,j}\mapsto :E_{i,j}:-\delta_{i,j}\theta(i> 0)\kappa$,
where $\theta(i> 0)=1$ if $i> 0$ and $0$ otherwise. 
The action of $\glinf$ on $\mathcal{Y}_{\al,\ga}$ 
can be extended to that of $\gl_{\infty,\kappa}$,
since acting with the latter on GZ patterns of hook shape 
only finitely many terms are produced.

Now let us turn to the 
Macmahon module $\mathcal{M}_{\al,\beta,\gamma}(u,K)$.
In the limit \eqref{limit}, the eigenvalues \eqref{PSIABG}
of $\psi^\pm(z)$ tend to 
\begin{align*}
\frac{1-K_1u/z}{1-q^{-\beta_1}u/z}
\prod_{i=1}^\infty
\frac{1-q^{i-\beta_i}u/z}{1-q^{i-\beta_{i+1}}u/z},
\end{align*}
where $K_1$ denotes the limiting value of $K$.
In order that this limit be $1$,  
we are forced to take 
$\beta_i=0$ for all $i$ and $K_1=1$.
Assuming this, 
consider the action of $e(z)$, $f(z)$
which we write in the form 
\begin{align*}
&e(z)|\boldsymbol{\la}\rangle=
\sum_{i,k=1}^\infty C^+_{i,k}(\boldsymbol{\la})
|\boldsymbol{\la}+1^{(k)}_i\rangle, 
\quad
f(z)|\boldsymbol{\la}\rangle=
\sum_{i,k=1}^\infty C^-_{i,k}(\boldsymbol{\la})
|\boldsymbol{\la}-1^{(k)}_i\rangle. 
\end{align*}
Let us compute the action of $e(z)$ in the limit $q_1\rightarrow1$.
We use \eqref{EACT} in the form
\begin{align*}
e(z)|\lab\rangle=\sum_{i,k=1}^\infty\frac1{1-q_1}\psi_{\la^{(k)},i}\psi^{(k-1)}_\lab(u/z)
\delta(q_1^{\la^{(k)}_i}q_3^{i-1}u_k/z)|\lab+1^{(k)}_i\rangle,
\end{align*}
where $\la^{(k)}_i=\mu^{(k)}_i-\al_k$ and $u_k=uq_1^{\al_k}q_2^{k-1}$.
We have for $q_1\to1$
\begin{align*}
&\frac{1}{1-q_1}\psi_{\la^{(k)},i}=
\begin{cases}
O(1)& (i\neq 1);\\
O\bigl(\frac{1}{1-q_1}\bigr)&(i=1),\\
\end{cases}\\
&\frac{1-q_3^{k-i}}{1-q_3^{1-i}}\psi^{(k-1)}_\lab(u/z)
\delta(q_1^{\la^{(k)}_i}q_3^{i-1}u_k/z)
=O(1),
\end{align*}
so that 
\begin{align*}
C^+_{i,k}(\boldsymbol{\la})=
\begin{cases}
O(1)& (i\neq k);\\
O\bigl(\frac{1}{1-q_1}\bigr) & (i=k).
\end{cases}
\end{align*}
Similarly, using
\begin{align*}
f(z)|\lab\rangle=\sum_{i,k=1}^\infty\frac{q_1}{1-q_1}\psi'_{\la^{(k)},i}\psi^{'(k+1)}_\lab(u/z)
\delta(q_1^{\la^{(k)}_i}q_3^{i-1}u_k/z)|\lab+1^{(k)}_i\rangle,
\end{align*}
we find
\begin{align*}
C^-_{i,k}(\boldsymbol{\la})=
\begin{cases}
O(1)& (i\neq k);\\
O\bigl(1-q_1) & (i=k).
\end{cases}
\end{align*}
Hence, passing to the new basis
$\ket{\boldsymbol\la}=(1-q_1)^{p(\boldsymbol\la)}\ket{\boldsymbol\la}^{new}$ 
where $p(\boldsymbol\la)=\sum_{i=1}^\infty\la^{(i)}_i$,    
the matrix coefficients for $e(z)$, $f(z)$ have well defined limits. 
Clearly the same is true about the quotient module
$\mathcal{N}^{n,n}_{\al,\emptyset,\gamma}(u)$.

Thus we have shown the first part of the following:

\begin{prop}\label{Nnn}
$(i)$ If $\beta=\emptyset$, then the Macmahon module 
$\mathcal{M}_{\al,\emptyset,\gamma}(u,q_2^nq_3^n)$
and its irreducible quotient
$\mathcal{N}^{n,n}_{\al,\emptyset,\gamma}(u)$
have well-defined limits. 

$(ii)$ Assume $\ga_1=\cdots=\ga_n=c$, $\kappa=n$.  
Then as $\mathfrak{d}_{q,\kappa,0}$-modules,
the limit of $\mathcal{N}^{n,n}_{\al,\emptyset,\gamma}(u)$ 
is isomorphic to the pullback 
$\iota^*_u\left(\mathcal{W}_{\theta^{(n)}(\al,c)}\right)$
of the $\mathfrak{gl}_{\infty,\kappa}$-module given in \eqref{GZ-hwt}. 
\end{prop}
\begin{proof}
Let us show (ii). 
By construction, the limit of 
$\mathcal{N}^{n,n}_{\al,\emptyset,\gamma}(u)$
and
$\iota^*_u\left(\mathcal{W}_{\theta^{(n)}(\al,c)}\right)$
both have bases 
labeled by the same combinatorial set, the GZ-pattern of hook type.  
It is easy to see that
$\iota^*_u\left(\mathcal{W}_{\theta^{(n)}(\al,c)}\right)$
is an irreducible $\mathfrak{d}_{q,\kappa,0}$-module. 
Hence 
it is sufficient to check that the lowest weights are the same.

The limit of $(\psi^+(z)-\psi^-(z))/(1-q_1)$
gives the eigenvalues of $\bar{h}_m$, which in turn 
gives those of $:E_{i,i}:$
via \eqref{lim-h1}, \eqref{lim-h2}
and \eqref{bar1}, \eqref{bar2}.
Denoting by $\theta_i$ the eigenvalues of $E_{i,i}$ we find
\begin{align}
\theta_{i}=
\begin{cases}
d_{1,-i+1}-\sum_{j=1}^\infty(d_{j,j-i+1}-d_{j+1,j-i+1})& 
(i\le 0);\\
\kappa_--\sum_{j=1}^\infty(d_{j+i-1,j}-d_{j+i,j})& (i>0),\\
\end{cases}
\label{theta}
\end{align}
where $d_{i,j}=\max(\ga_i,\al_j)$ and $K=q_1^{\kappa_-}$. 
In the case 
$\kappa_-=-n$, $\al_j=\gamma_j=0$ ($j>n$), f
$\ga_1=\cdots=\ga_n=c$ 
and $\al_1\ge\cdots\ge\al_k
\ge c\ge \al_{k+1}\ge\cdots\ge \al_n$, this reduces to formula
\eqref{GZ-hwt}.
\end{proof}

\section
{Characters}\label{characters}
All $\mathcal E$-modules in this paper are graded by convention $\deg
e_i=-\deg f_i=1$, $\deg \psi^\pm_i=0$. Computation of the
characters of $\mathcal N^{m,n}_{\al,\beta,\gamma}$
is a very interesting and challenging problem. It looks
that in a lot of cases there are many seemingly unrelated highly
non-trivial formulae.  

In this section we compute
the characters of $\mathcal{N}^{n,n}_{\al,\emptyset,\emptyset}(u)$.
Note that $\gamma=\emptyset$.
In this case, by Proposition \ref{Nnn}, our problem is equivalent to computing
the characters of $\gl_\infty$ modules $W_{\theta^{(n)}(\alpha,0)}$
with the degree defined by $\deg E_{ij}=j-i$.
We also present several conjectures at the end.
\subsection{Bosonic construction} 
In this subsection we follow \cite{KR2}.
The main tool is the bosonic construction of $\gl_\infty$ modules.
We omit proofs when they are available in \cite{KR2}.

Let $H$ be the algebra generated by generators $d_i,d_i^*$, $i\in\Z$ with defining relations
\be
[d_i,d_j]=[d_i^*,d_j^*]=0,\qquad [d_i^*,d_j]=\delta_{i,-j}.
\en
Let $U$ be the cyclic representation of the algebra $H$ with the cyclic vector $v$ satisfying
\be
d_{i+1}v=d_{i}^*v=0,
\qquad i\in\Z_{\geq 0}.
\en
The following lemma is clear.

\begin{lem}\label{irred} The module $U$ is an irreducible $H$ module. \qquad $\Box$
\end{lem}

Introduce the notation 
\be
:d_id_{-i}^*:=
\begin{cases}
 d_id_{-i}^* & (i\le 0);\\
 d_{-i}^*d_i & (i> 0).\\
\end{cases}
\en

Define an action of $\gl_1=\C\cdot e_{11}$ in $U$ by
\be
e_{11}=\sum_{i\in\Z}:d_id_{-i}^*:.
\en
Define an action of $\gl_\infty$ in $U$ by letting the generator $E_{ij}$ act as
\be
E_{ij}=d_id_{-j}^*.
\en

\begin{prop}\label{sw0}
The actions of $\gl_1$ and $\gl_\infty$ in $U$ commute.
We have the decomposition of $\gl_{\infty}$ modules
\be
U=\oplus_{k\in \Z}  W_{\theta^{(1)}(k,0)}.
\en
Moreover, $W_{\theta^{(1)}(k,0)}=\{v\in U\ |\ e_{11}v=kv\}$.
The module $W_{\theta^{(1)}(k,0)}$ is
the irreducible lowest weight $\gl_\infty$ module with the lowest weight
$\theta^{(1)}(k,0)$ given by \eqref{GZ-hwt} and
with the lowest weight vector 
$d_{0}^kv$ if $k> 0$
and
$(d_{-1}^*)^{-k}v$ if $k\le 0$.
\end{prop}
Let $H_n=H^{\otimes n}$. We denote the generators of $H_n$ by
$d_i^{(k)},d_i^{(k)*}$, $i\in\Z,k\in\{1,\dots,n\}$.
We have
\be
[d_i^{(k)},d_j^{(l)}]=[d_i^{(k)*},d_j^{(l)*}]=0,\quad [d_i^{(k)},d_j^{(l)}]=\delta_{i,-j}\delta_{k,l}.
\en
Then $U_n=U^{\otimes n}$ is naturally an $H_n$-module. By Lemma \ref{irred}, $U_n$ is an irreducible
$H_n$-module. We set $v_n=v^{\otimes n}$.

Define an action of $\gl_n$ in $U_n$ by letting the matrix units $e_{kl}$ act as
\be
e_{kl}=\sum_{i\in\Z}:d_i^{(k)}d_{-i}^{(l)*}:, \qquad k,l=1,\dots,n.
\en
Define an action of $\gl_\infty$ in $U_n$ by letting the matrix units $E_{ij}$ act as
\be
E_{ij}=\sum_{k=1}^nd_i^{(k)}d_{-j}^{(k)*}, \qquad i,j\in\Z.
\en

Proposition \ref{sw} below generalizes Proposition \ref{sw0} to the case
where $\gl_1$ is replaced with $\gl_n$. It was proved in \cite{KR2}, see also \cite{W}.
It is a $\gl_\infty$ version of the Schur-Weyl-Howe duality. The latter states
\begin{prop}\label{DUALITY}
Let $N$ be an integer such that $N\geq n$.
In the above setting consider the subspace
\begin{align*}
U_{n,N}=\C[d^{(k)}_i;1\leq k\leq n,-N+1\leq i\leq0]v_n\subset U_n.
\end{align*}
We have mutually commutative actions of $\gl_n$ and $\gl_N\simeq\gl_{-N+1,0}$ on $U_{n,N}$,
and with respect to these actions, we have the decomposition
\begin{align*}
U_{n,N}=\oplus_\al L_\al\otimes \tilde L_\al,
\end{align*}
where $\alpha=(\alpha_1,\alpha_2,\dots,\alpha_n,0,0,\ldots)$, $\alpha_i\in\Z$ such that
$\al_1\geq\al_2\geq\ldots\geq\al_n\geq0$, and $L_\al$ $($resp., $\tilde L_\al)$
is the irreducible $\gl_n$ $($resp., $\gl_N)$ module with the highest weight $(\al_1,\ldots,\al_n)$
$($resp., the lowest weight $(0,\ldots,0,\al_n,\ldots,\al_1))$.
The component $L_\al\otimes \tilde L_\al$ is generated by the cyclic vector
\begin{align*}
v^{(n,N)}_\alpha=\prod_{i=1}^{n-1}(D_i)^{\alpha_i-\alpha_{i+1}}(D_n)^{\alpha_n},
\end{align*}
where $D_i$ are given by
\be
D_i=\det(d_{-j+1}^{(l)})_{j,l=1,\dots,i}.
\en
\end{prop}
Similarly we define
\be
D_i^*=\det(d_{-j}^{(n+1-l)*})_{j,l=1,\dots,i}.
\en

Now we give the duality statement for $\gl_n$ and $\gl_\infty$.
\begin{prop}\label{sw} 
We have the decomposition 
\be
U_n=\oplus_\al
\Bigl(L_\alpha\otimes W_{\theta^{(n)}(\alpha,0)}\Bigr),
\en 
where $\alpha=(\alpha_1,\alpha_2,\dots,\alpha_n)$
$(\alpha_1\geq\alpha_2\geq\dots\geq\alpha_n,\alpha_i\in\Z)$, and
$L_\alpha$ is the irreducible $\gl_n$ module with highest weight
$\alpha$
and $W_{\theta^{(n)}(\alpha,0)}$ is the irreducible lowest weight $\gl_\infty$ module
with lowest weight $\theta^{(n)}(\alpha,0)$ given by \eqref{GZ-hwt}.
Moreover, $L_\alpha\otimes W_{\theta^{(n)}(\alpha,0)}$ is generated by
\begin{align}\label{sing vect2}
v_\alpha=\prod_{i=1}^{k(\al)-1}
(D_i)^{\alpha_i-\alpha_{i+1}} 
(D_{k(\al)})^{\alpha_{k(\al)}} 
\times 
\prod_{i=1}^{s(\al)-1}(D_i^*)^{\alpha_{n-i}-\alpha_{n-i+1}}
(D_{s(\al)}^*)^{-\alpha_{n-s(\al)+1}}\ v_n,
\end{align}
where $k(\al)$, $s(\al)$ are the numbers of positive and negative parts of $\alpha$ respectively.
\end{prop}

\subsection{Characters of $W_{\theta^{(1)}(k,0)}$}
Consider the set
\be
C_a=\{(\la,\mu)\ |\ \la,\mu - \text{partitions},\ \mu_1+a\geq \la_1\}.
\en
Let 
\be
\bar \chi_a=\sum_{(\la,\mu)\in C_a}q^{|\la|+|\mu|}
\en
be the corresponding formal character. 

Set 
\be
(q)_\infty=\prod_{i=1}^\infty(1-q^i).
\en

\begin{lem}\label{rec lem} For $k\in\Z_{\geq 0}$ we have the recursive relation
\be
\bar\chi_k(q)+q^{k+1}\bar\chi_{k+1}(q)=\frac{1}{(q)_\infty^2}.
\en
\end{lem}
\begin{proof}
We construct a map 
\be
\iota_k:\ C_k\sqcup C_{k+1}\to C_{\infty}:=\{(\la,\mu)\ |\ \la,\mu - \text{partitions}\}
\en
as follows. For $(\la,\mu)\in C_k$ we set $\iota_k(\la,\mu)=(\mu,\la)$.
 For $(\la,\mu)\in C_{k+1}$ we set 
$\iota_k(\la,\mu)=(\tilde{\mu},\tilde{\la})$,
where
\be
\tilde{\la} 
=(\mu_1+k+1,\la_1,\la_2,\la_3\dots),\qquad 
\tilde{\mu} 
=(\mu_2,\mu_3,\mu_4,\dots).
\en
Clearly, $\iota_k$ is a bijection. The lemma follows.
\end{proof}
\begin{cor} For $k\in\Z_{\geq 0}$ we have
\be
\bar\chi_k=\frac{1}{(q)_\infty^2}\sum_{j=0}^\infty (-1)^jq^{j(j+1)/2+jk}.
\en
\end{cor}
\begin{proof}
Repeating using Lemma \ref{rec lem}, we compute
\be
\bar\chi_k(q)=\frac{1}{(q)_\infty^2}-q^{k+1}\bar\chi_{k+1}(q)
=\frac{1}{(q)_\infty^2}-q^{k+1}\Bigl(\frac{1}{(q)_\infty^2}-q^{k+2}\bar\chi_{k+2}(q)\Bigr)=...\ .
\en
Continuing, we obtain the corollary.
\end{proof}

For  $k\in\Z_{\geq 0}$, set

\be
\chi_k=\bar\chi_k,\qquad \chi_{-k}=q^k\bar\chi_k.
\en

We set $\deg d_{i}=\deg d_{i}^*=-i$, $\deg v=0$.
\begin{cor}
For $k\in\Z$, we have
\be
\chi(W_{\theta^{(1)}(k,0)})=\chi_k.
\en
\end{cor}
\begin{proof}
By Proposition \ref{GLMODULE}, for $k\in\Z_{\geq0}$
the modules $W_{\theta^{(1)}(k,0)}$ and
$W_{\theta^{(1)}(-k,0)}$ have bases parameterized by the set $C_k$, therefore
the corollary follows.
\end{proof}
\subsection{Characters of $W_{\theta^{(n)}(\alpha,0)}$}
Set $\deg d_{i}^{(j)}=\deg d_{i}^{(j)*}=-i$, $\deg v_n=0$. 
Proposition \ref{sw} allows us to compute the character of $W_{\theta(\alpha,0)}$ in terms of $\chi_k$.

For $\alpha=(\al_1,\ldots,\al_n)$ $(\al_1\geq\ldots\geq\al_n,\al_i\in\Z)$,
let $k(\alpha)$ be the number of positive parts of $\alpha$. We set
\be
p(\alpha)=\sum_{i=1}^{k(\alpha)}(i-1)\alpha_i+\sum_{i=1}^{n-k(\alpha)}i\alpha_{n-i+1}.
\en

Then $p(\alpha)$ is the degree of the singular vector $v_\alpha$ given by \Ref{sing vect2}.

Recall that the $\gl_n$ weight $\rho$ is given by 
\be
\rho=\biggl(\frac{n-1}{2},\frac{n-3}{2},\dots,\frac{1-n}{2}\biggr),
\en
and that the symmetric group $S_n$ acts on the 
$\gl_n$ weights by simply permuting the indexes.
\begin{thm} We have
\be
q^{p(\alpha)}\chi(W_{\theta(\alpha,0)})=
\sum_{\sigma\in S_n} (-1)^\sigma \prod_{i=1}^n \chi_{(\sigma(\al+\rho)-\rho)_i}.
\en
\end{thm}
\begin{proof}
The character of the subspace of vectors in $U_n$ of the $\gl_n$ weight $\mu=(\mu_1,\ldots,\mu_n)$
is obviously given by
$\prod_{i=1}^n\chi_{\mu_i}$. By Proposition \ref{sw}, the space $W_{\theta(\alpha,0)}$ is identified with the space of $\gl_n$ singular vectors of weight $\al$ in $U_n$ with the shift of the degree given by $p(\alpha)$. Moreover, Proposition \ref{sw} asserts that $U_n$ is a direct sum of finite-dimensional $\gl_n$-modules. Then
the character of the space of $\gl_n$ singular vectors of weight $\al$ is computed as the alternating sum of the characters of the weight subspaces. 
\end{proof}
\subsection{Other character formulas}
We finish with some conjectures which we checked for the small values of parameters.
\begin{conj}
The character of $\mathcal N_{0,0,0}^{(1,m)}$ is given by
\be
\chi(\mathcal N_{0,0,0}^{(1,m)})=\frac{\prod_{i=1}^{m-2}(1-q^i)^{m-i-1}}{(q)^{m+1}_\infty}\sum_{j=0}^\infty(-1)^jq^{\frac12 j(j+1)}\prod_{i=1}^{m-1}(1-q^{i+j}).
\en
\end{conj}
\begin{conj}
The character of $\mathcal N_{0,0,0}^{(n,m)}$, where $n\geq m$ is given by
\begin{align*}
\chi(\mathcal N_{0,0,0}^{(n,m)})= \frac{1}{(q)_\infty^{m+n}}\sum_{\la_1\geq\la_2\geq \dots\geq \la_m\geq 0}(-1)^{\sum_{i=1}^m \la_i}q^{\frac12 \sum_{i=1}^m(\la_i^2+ (2i-1)\la_i)}\times \\
\prod_{1\leq i<j\leq m}(1-q^{\la_i-\la_j+j-i})\prod_{1\leq i<j\leq n}(1-q^{\la_i-\la_j+j-i}).
\end{align*}
Here we set $\la_j=0$ if $j>m$.
\end{conj}

{\bf Acknowledgments.}
Research of BF is partially supported by
RFBR initiative interdisciplinary project grant 09-02-12446-ofi-m,
by RFBR-CNRS grant 09-02-93106, RFBR grants 08-01-00720-a, 
NSh-3472.2008.2 and 07-01-92214-CNRSL-a.
Research of MJ is supported by the Grant-in-Aid for Scientific
Research B-23340039.
Research of EM is
supported by NSF grant DMS-0900984. 
The present work has been carried out during the visits of BF and EM 
to Kyoto University. They wish to thank the University for hospitality.

\end{document}